\documentclass[16pt]{article}

\usepackage{graphicx,mathrsfs,epstopdf,subfigure}
\usepackage{amsmath,amsfonts,amssymb,amsthm}
\usepackage[colorlinks,citecolor=blue,urlcolor=blue]{hyperref}
\usepackage{algorithmic}

\usepackage{tikz}
\usepackage{verbatim}
\usetikzlibrary{positioning}
\usepackage{bbm}
\usepackage[margin=1.2in]{geometry}
\usepackage{setspace}
\usepackage{fancyhdr}
\usepackage[ruled, 
lined, 
commentsnumbered]{algorithm2e}

\numberwithin{equation}{section}

\newtheorem{theorem}{Theorem}[section]
\newtheorem{lemma}[theorem]{Lemma}
\newtheorem{proposition}[theorem]{Proposition}
\newtheorem{corollary}[theorem]{Corollary}

\newtheorem{conjecture}[theorem]{Conjecture}

\theoremstyle{definition}
\newtheorem{definition}[theorem]{Definition}
\newtheorem{example}[theorem]{Example}
\newtheorem{remark}[theorem]{Remark}

\newcommand{\QED}{\ifhmode\unskip\nobreak\fi\quad {\rm Q.E.D.}} 

\newcommand{\RR}{\mathbb{R}}
\newcommand{\R}{\mathbb{R}}

\newcommand{\Z}{\mathbb{Z}}

\newcommand{\mtp}{\text{MTP}_2}
\newcommand{\conv}{\mathrm{conv}}
\newcommand{\MMconv}{\mathrm{MMconv}}
\usepackage{authblk}

\title{Maximum Likelihood Estimation for \\ Totally Positive Log-Concave Densities}
\date{}
\author[1,2]{\uppercase{Elina Robeva}}
\author[3,4]{\uppercase{Bernd Sturmfels}}
\author[5]{\uppercase{Ngoc Tran}}
\author[1]{\uppercase{Caroline Uhler}}
\affil[1]{Massachusetts Institute of Technology}
\affil[2]{University of British Columbia}
\affil[3]{University of California at Berkeley}
\affil[4]{Max Planck Institute Leipzig}
\affil[5]{University of Texas at Austin}
\begin{document}
\maketitle

\begin{abstract}
 We study nonparametric maximum likelihood 
estimation for two classes of
multivariate distributions that imply strong forms of positive dependence;
namely log-supermodular (\emph{MTP$_2$}) distributions
     and  \emph{log-$L^\natural$-concave} (\emph{LLC}) distributions. 
In both cases
     we also assume log-concavity in order to ensure boundedness of the likelihood function. Given $n$ independent and identically distributed random vectors from one of our distributions, 
  the maximum likelihood estimator (MLE) exists a.s.~and is unique a.e.~with probability one when $n\geq 3$.
  This holds independently of the ambient dimension $d$.
  We conjecture that the MLE is always the exponential of a tent function.
  We prove this result for samples in $\{0,1\}^d$
    or in $\mathbb{R}^2$ under MTP$_2$, and
    for samples in $\mathbb{Q}^d$      under LLC.
         Finally, we provide a conditional gradient algorithm for computing the maximum likelihood estimate.
         \end{abstract}

\linespread{1.5}


\section{Introduction}
Let $X_1,\dots , X_n\in\mathbb{R}^d$ be independent and identically distributed (i.i.d.) samples from a distribution with density function $f_0$. 
Nonparametric methods are attractive for computing  an estimate for $f_0$ since they do
not impose any parametric assumptions. Popular such methods
include kernel density estimation, adaptive smoothing and neighbor-based techniques. For details we refer to the
surveys \cite{izenman1991review,turlach1993bandwidth, scott2015multivariate, silverman2018density,wand1994kernel,chen2017tutorial,wasserman2016topological}
 and references therein. These techniques
 require choosing a smoothing parameter, either in the form of a bandwidth for kernel density estimation, or a regularization penalty and clustering parameters.

An approach gaining popularity in recent years that does not require the choice of a tuning parameter is shape-constrained density estimation. 
Here one seeks to maximize the (possibly weighted) log-likelihood function
$$\ell(f)\,\, =\,\,\sum_{i=1}^n w_i\log f(X_i),$$
subject to a shape constraint  $f\in\mathcal{F}$. The weights $w_i$ are fixed, positive, and satisfy $\sum_{i=1}^n w_i = 1$. They can be interpreted as the relative importance or confidence among different samples. If the class $\mathcal{F}$ of shapes is unrestricted, i.e., it contains all density functions $f : \R^d \to \R$, then the maximum likelihood estimator (MLE) does not exist. The likelihood function is unbounded, 
even if the density function is constrained to be unimodal.
Thus, it is of interest to identify shape constraints that restrict $\mathcal{F}$ sufficiently to make the MLE well-defined with good numerical and theoretical properties, but at the same time keep $\mathcal{F}$ large enough to be
relevant for applications.

To this end, different shape constraints have been considered: monotonicity (studied in \cite{Grenander}
for $d=1$ and extended in \cite{Polonik} to $d > 1$), convexity (studied in \cite{Groeneboom}
for $d=1$ and extended in \cite{Seregin} to $d > 1$), and, most prominently, log-concavity.
The log-concave MLE  $\hat{f}_n$ was first introduced by Walther~\cite{Walther_2002} and has been studied
in great detail recently. 
In particular, the MLE exists and is almost everywhere unique with probability one when $n\geq d+1$~\cite{CSS, Duembgen}. Furthermore, Cule, Samworth and Stewart~\cite{CSS} showed that 
 $\log(\hat{f}_n)$ is a piecewise linear concave function;
 namely, it is a \emph{tent function} with tent poles at the observations $x_i\in\RR^d$. Such a tent function induces a regular subdivision of 
 the configuration $X=\{x_1, \ldots ,x_n\}\subset \RR^d$, and in fact any regular subdivision can arise 
 with positive probability as the MLE for some set of weights~\cite{RSU}. For a recent review of 
 shape-constrained density estimation see Groeneboom and Jongbloed~\cite{GJ_review}.

In this paper, we consider the problem of maximum likelihood estimation under \emph{total positivity}, a special form of positive dependence. A distribution defined by a density function $f$ over $\mathcal{X}=\prod_{i=1}^d\mathcal{X}_{i}$, where each set $\mathcal{X}_i$ is totally ordered, is \emph{multivariate totally positive of order 2} ($\mtp$) if
$$
f(x)f(y)\quad\leq \quad f(x\wedge y)f(x\vee y)\qquad\mbox{for all }x,y\in \mathcal{X},
$$
where $x\wedge y$ and $x\vee y$ are the element-wise minimum and maximum. We restrict our attention to densities on $\mathcal X = \mathbb R^d$. 
If $f$ is strictly positive, then $f$ is $\mtp$ if and only if $f$ is \emph{log-supermodular}, i.e., $\log(f)$ is supermodular. $\mtp$ was introduced in~\cite{fortuin1971correlation}. It implies \emph{positive association}, an important property in probability theory and statistical physics, which is usually difficult to verify. In fact, most notions of positive dependence are implied by $\mtp$; see for example~\cite{colangelo2005some} for a recent overview. The special case of Gaussian $\mtp$ distributions was studied by Karlin and Rinott~\cite{karlinGaussian} and in~\cite{slawski2015estimation, LUZ} from 
the perspective of MLE and optimization. More recently, estimating two-dimensional smooth MTP$_2$ densities was studied in \cite{HMRR2}, and estimating supermodular matrices was studied in~\cite{HMRR1}. 

Binary MTP$_2$ distributions have long been known to have connections with statistical mechanics \cite{fortuin1971correlation}, computer storage \cite{van1992lru}, item response theory \cite{bartolucci2000likelihood,junker2001nonparametric,ligtvoet2015test}. More recently, they have been used for analyzing psychological disorders \cite{lauritzen2019total}.
By \cite[Proposition 7.4]{murota2003discrete}, a probability mass function $f$ on $\{0,1\}^d$ is MTP$_2$ if and only if $-\log(f)$ is a submodular set function. Since the set of submodular set functions is convex, the MLE problem under the MTP$_2$ constraint for binary variables is a convex optimization problem over a convex set. \cite{bartolucci2000likelihood} prove this under a linear reparametrization of $-\log(f)$, and propose to use the constrained Fisher scoring algorithm to find the MLE in the binary case. 

Unfortunately, in the continuous case maximum likelihood estimation under $\mtp$ is ill-defined, since the likelihood function is unbounded. For this reason, we consider the problem of nonparametric density estimation, where $\mathcal{F}$ consists of all \emph{$\mtp$ and log-concave} density functions $f:\RR^d\to\RR$. A strictly positive function $f:\RR^d\to\RR$ is log-concave if $\log(f)$ is a concave function, i.e.,  \begin{equation}
 \label{eq:loco} f(x)^\lambda f(y)^{1-\lambda}\, \leq \, f \Big(\lambda x + (1-\lambda) y\Big) \quad \hbox{ for all $x,y \in \RR^d, \lambda \in [0,1]$.} 
 \end{equation}
The class $\mathcal{F}$ of log-concave $\mtp$ densities contains many interesting distributions, such as totally positive Gaussians~\cite{LUZ}.
These include Brownian motion tree models, which are widely used to model evolutionary processes~\cite{felsenstein_maximum-likelihood_1973}. Totally positive Gaussians have also been used for covariance estimation in portfolio selection \cite{agrawal2019covariance}. The densities in our class $\mathcal{F}$ have various desirable properties; in particular, the class $\mathcal{F}$ is closed under marginalization and conditioning~\cite{KarRin80, MTP2Markov2015}.

We also consider a subclass of $\mtp$ distributions where the logarithm of the density function is  \emph{$L^\natural$-concave}. A function $g:\mathbb{R}^d\to \mathbb{R}$ is $L^\natural$-concave~if
$$g(x)+g(y)\leq g((x+\alpha \mathbf{1})\wedge y) + g(x\vee (y-\alpha \mathbf{1})) \quad \textrm{for all } \alpha\geq 0 \textrm{ and } x,y\in \mathbb{R}^d,$$
where $\mathbf{1}$ denotes the all ones vector~\cite{murota2003discrete,murota2009recent}.
The set $\mathcal{L}$ of log-$L^\natural$-concave (LLC) densities is an appealing subclass of MTP$_2$ densities. For example, a $d$-dimensional Gaussian distribution is MTP$_2$ if and only if its inverse covariance matrix $K$ is an {\em M-matrix}, i.e. $K_{ij}\leq 0$ for all $ i\neq j$, and it is
LLC if and only if $K$ is also \emph{diagonally dominant}, i.e.
$$K_{ij}\leq 0 \,\textrm{ for all \,} i\neq j \quad \textrm{and}\quad  \sum_{j=1}^d K_{ij}\geq 0 
\,\textrm{ for all }\, i=1,\ldots d.$$
For Gaussian graphical models, if the density is LLC, loopy belief propagation converges and hence marginal distributions can be computed efficiently \cite{LBP_inverse_M,walk_summable}. 
Such functions also appear naturally in economics~\cite{tamura2004applications,danilov2001discrete,fujishige2003note}, network flow problems~\cite{narayanan1997submodular},
phylogenetics~\cite{Zwiernik} and combinatorics~\cite{Joswig2010tropical,lam2007alcoved}. In these cases, the $i$-th coordinate of a vector $x$ may represent the unit price of item $i$ or the potential at node $i$. Since for such applications only the coordinate differences $x_i-x_j$ matter, $L^\natural$-concavity is the appropriate notion to consider. Another example is the case of max-linear graphical models~\cite{wang2011conditional, gissibl2018max,kluppelberg2019bayesian}, where the variables of interest are defined by recursive max-linear equations and are invariant under scaling by constants. The distribution of their logarithms is invariant under translations, in other words, they are LLC. Finally, when the points have integer coordinates, $L^\natural$-concavity is equivalent to discrete concavity \cite[\S 1]{murota2003discrete}. This notion
is important in combinatorial optimization \cite{murota2003discrete,murota2009recent}.
 
Our aim is to study the MLE in two settings: log-concavity combined with $\mtp$ and log-concavity combined with LLC. In other words, we study properties of the solutions to the optimization problems
\begin{align}\label{eq:optproblem}
\text{maximize} & \quad\quad 
w_1 \log f(x_1) + \cdots + w_n \log f(x_n) \\
\text{such that} &\quad\quad f \text{ is a log-concave and MTP}_2 \text{ density,}\notag
\end{align}
and
\begin{align}\label{eq:optproblem.ell}
\text{maximize} & \quad\quad
w_1 \log f(x_1) + \cdots + w_n \log f(x_n) \\
\text{such that} &\quad\quad f \text{ is a log-concave and LLC density.}\notag
\end{align}

\noindent  \textbf{Results.}  Our main result, Theorem \ref{thm_ex_unique}, concerns the existence and uniqueness of the MLE almost surely, and the minimum number of samples needed for this to happen. 
If the support of the original underlying density is full-dimensional, that number is
three for~$\mtp$, and two for LLC. This does not depend on the dimension $d$. 
In addition, we show that the MLE is unique and consistent. The proof is given in Appendix~\ref{app_1}.
It builds on work of Royset and Wets~\cite{royset2017constrained} which provides a general framework for proving existence and consistency of the MLE in shape constrained density estimation.
  A direct application of~\cite{royset2017constrained} to log-concave $\mtp$ or LLC densities would require an extra technical assumption (see \cite[Proposition 4.6]{royset2017constrained}). Our theorem requires no such assumption, and computes the smallest possible number of samples needed for the MLE to exist.

\begin{theorem}[Existence, uniqueness, and consistency of the MLE]\label{thm_ex_unique}
Let $X_1, \dots, X_n$ be i.i.d samples from a continuous distribution with density $f_0$. The following hold with probability one:
\begin{itemize}
	\item If $n \geq 3$, the $\mtp$ log-concave MLE exists and is unique almost everywhere. 
	\item If $n \geq 2$, the LLC log-concave MLE exists and is unique almost everywhere. 
\end{itemize}
In both cases, when the MLE $\hat{f}_n$ exists, it is consistent in the sense that
 $\hat{f}_n$ converges almost surely to $f_0$ in the Attouch-Wets metric of \cite{royset2017constrained}.\end{theorem}
 
 \begin{remark}The Attouch-Wets metric, also known as the hypo-distance metric, was defined in \cite[Section 3.1]{royset2017constrained} to metrize convergence of hypographs of density functions. As shown in \cite[Proposition 3.2, 3.3]{royset2017constrained}, under certain assumptions which apply to a large class of functions, convergence under the Attouch-Wets metric implies convergence under the Hellinger and $L^2$ distances.
 \end{remark}

The other results of our paper concern the shape and computation of the MLE. We describe the support of the MLE and give algorithms for computing it (cf. Section~\ref{sec_2}). We provide conditions on the samples $X$ which ensure that the MLE under $\mtp$ and LLC, respectively, can be computed by solving a convex optimization problem. Under these conditions, which we call {\em tidy}, the $\mtp$ and 
LLC MLEs behave like the log-concave MLE. They are piecewise-linear and can be computed by solving a finite-dimensional convex optimization problem (cf. Theorem \ref{thm:tidy_solution}). For $\mtp$, being tidy includes $X\subset \RR^2$ or $X\subseteq \{0,1\}^d$, and for LLC, it includes $X \subset \mathbb{Q}^d$. Since numerical computations are usually performed in $\mathbb{Q}$ by rounding real numbers, the LLC MLE can always be computed 
numerically using the optimization problem in Theorem \ref{thm:tidy_solution}. Building on the subgradient descent algorithm for computing the log-concave MLE  in~\cite{CSS}, we propose a conditional gradient method for computing the MTP$_2$ log-concave MLE for tidy configurations.

For non-tidy configurations, we conjecture that the MLE is piecewise linear as well, and
can be computed via the finite dimensional convex program used for tidy configurations (cf. Conjectures \ref{conj_main} and \ref{conj_simple}).
 As steps towards proving these conjectures, we show that a tent function is concave and $\mtp$ if and only if it induces a bimonotone subdivision. This result (cf. Theorem~\ref{thm:BimonotoneSubdivision}) is the analogue of \cite[Theorem 7.45]{murota2003discrete} for the LLC case. Finally, Section~\ref{sec_5} raises
 research  questions in geometric combinatorics and discrete convex analysis that are of independent interest. 
 
\medskip

\textbf{Organization.} Our paper is organized as follows. Section~\ref{sec_2}  characterizes the support of the MLE under $\mtp$ and LLC respectively,
and it offers algorithms for computing these. 
Sections~\ref{sec_3} develops the convex optimization problem associated to the tidy case, and Section~\ref{sec_4} develops algorithms for solving this optimization problem. Section~\ref{sec_5} analyzes the estimation problem in the general case. We conclude with a short discussion in Section~\ref{sec_6}. 

\section{Support of the MLE}
\label{sec_2}

According to~\cite{CSS}, the log-concave MLE exists with probability one if there are at least $d+1$ samples, and its support is the convex polytope $\conv(X)$. The support of the log-concave and MTP$_2$/LLC estimates always contains $\conv(X)$, but is in general larger. In this section we develop the relevant geometric theory to compute it.

We begin with the log-concave MTP$_2$ case. In the course of the proof of Theorem~\ref{thm_ex_unique} (see Appendix~\ref{app_1}), we show that if an MLE exists,
 then its support is the {\em min-max convex hull} of the samples $x_1,\ldots,x_n$ in $\RR^d$.

\begin{definition} 
A subset of $\mathbb R^d$ is {\em min-max closed} if, for any two elements $u$ and $v$ in the set, the vectors $u\wedge v$ and $u\vee v$ are also in the set. For a finite set $X \subset \R^d$, its {\em min-max closure} $\overline{X}$ is the smallest min-max closed set that contains $X$. Its {\em min-max convex hull} MMconv$(X)$ is the smallest min-max closed and convex set containing $X$. In discrete geometry~\cite{felsner2011distributive},
min-max closed convex sets are also known as {\em distributive polyhedra}.
\end{definition}

For a finite subset $X$ of $\RR^d$, 
the min-max closure $\overline{X}$ can be computed by adding points iteratively. That is, 
we set $X^{(0)} = X$, and we define
\begin{equation}
\label{eq:Xicompletion}
 X^{(i)} \,=\, \{u \wedge v\,: \,u,v\in X^{(i-1)}\}\,\cup \, \{u \vee v\,:\, u,v\in X^{(i-1)} \} \;\;\textrm{ for } i\geq 1.
 \end{equation}
 Since the $j$-th coordinate of each point in $X^{(i)}$  is among the $j$-th coordinates of the points in $X$, 
 equation (\ref{eq:Xicompletion}) defines an increasing nested sequence of sets 
 which stabilizes in finitely many steps, and the final set is $\overline{X}$. Note that to compute $\overline{X}$ we use a more efficient approach, described in Algorithm~\ref{alg:min-max_closure}.
 
 Suppose that  $X \subset \RR^d$ is finite. Then
  ${\rm MMconv}(X)$ is a convex polytope. At first, one might think that 
     ${\rm MMconv}(X) = \conv(\overline X)$. But this is true only when the dimension $d$ equals $2$.
For $d\geq 3$, MMconv$(X)$ is generally larger
than $\conv(\overline X)$. We refer to Example~\ref{ex:17} and \cite[Example 17]{QT}.

 The {\em 2-D Projections Theorem} below gives the linear inequality representation of 
 the polytope $\MMconv(X)$. This result was published in the 1970's by Baker and Pickley \cite{BP} and Topkis \cite{Topkis}, but they attribute it to George Bergman, who discovered it in the context of universal algebra. This theorem was extended to a characterization of min-max closed polytopes by Queyranne and Tardella \cite{QT}. For $i,j \in \{1, \dots, d\},i \neq j$, let $\pi_{ij}:\mathbb R^d\to\mathbb R^2$ denote the projection map onto the $i$-th and $j$-th coordinate.

\begin{theorem}[2-D Projections Theorem] \label{2DProjectionsTheorem}
\!\!For any finite subset  $X\!\subset \!\mathbb R^d$, 
$$ {\rm MMconv}(X) \,\,\,= \,\,\, \bigcap_{i\neq j}  \,
\pi_{ij}^{-1}\bigl( \,{\rm conv} \bigl(\, \pi_{ij} ({\overline X})\, \bigr) \,\bigr).
$$
\end{theorem}

{
	\begin{algorithm}[b!]
		\label{alg:QT}
		\caption{Computing  the min-max convex hull}
		\LinesNumbered
		\DontPrintSemicolon
		\SetAlgoLined
		\SetKwInOut{Input}{Input}
		\SetKwInOut{Output}{Output}
		\Input{A finite set of points $X$ in $\RR^d$.}
		\Output{The polytope ${\rm MMconv}(X)$ in $\RR^d$.}
		\BlankLine
		\For{each pair of distinct indices $i,j$}
		{Compute the convex hull  of the planar point configuration $\pi_{ij}({X})\cup\{\min(\pi_{ij}(X)),\max(\pi_{ij}(X))\}$;\;
			List the bimonotone inequalities that minimally describe this polygon in $\RR^2$;}	
		Collect all bimonotone inequalities. Their solution set in $\RR^d$ is ${\rm MMconv}(X)$; \;
	\end{algorithm}
}

Theorem \ref{2DProjectionsTheorem} reduces the computation of min-max convex hulls to the case $d=2$, and yields Algorithm~\ref{alg:QT} for computing MMconv$(X)$ for a finite set $X$. Note that to compute MMconv$(X)$ for a set of points $X\subseteq\mathbb R^2$, we can use the simple fact that MMconv$(X) = $ conv$(X\cup \{\min(X), \max(X)\})$, where $\min(X)$ and $\max(X)$ are the coordinatewise minimum and maximum of $X$. Therefore, Algorithm~\ref{alg:QT} boils down to computing $\binom d2$ convex hulls of two-dimensional sets of $n+2$ points each. So, its complexity is $ O(\binom d2 n\log n)$.

Note that each edge of 
 a min-max closed polygon
 in $\RR^2$ has non-negative slope, so the line it spans has the form
$\{a_1 z_1 + a_2 z_2 = b\}$ where $a_1 a_2 \leq 0$.
Extending this to higher dimensions $d \geq 2$, we say that a linear inequality on $\RR^d$ is 
{\em bimonotone} if it has the form $\, a_i z_i + a_j z_j \geq b$, where $a_ia_j \leq 0$.
Theorem~\ref{2DProjectionsTheorem} implies the following representation for
min-max closed polytopes.

\begin{corollary}[Theorem 5 in \cite{QT}] \label{cor:bimonotone} A polytope $P$ in $\mathbb R^d$ is 
min-max closed if and only if $P$ is defined by a set of
 bimonotone linear inequalities.
\end{corollary}

\begin{example} \label{ex:17} \rm 
Set $X = \{(0,0,0), (6,0,0), (6,4,0), (8,4,2)\}\subset\mathbb R^3$. 
This set is a chain in the coordinatewise 
partial order on $\mathbb R^3$, thus $\overline X = X$. Hence,
${\rm conv}(X) = {\rm conv}(\overline X)$ is a tetrahedron
with the four given vertices. 
This tetrahedron is not min-max closed: $(6,4,0)$ and $(6,3,3/2)$ are both  in ${\rm conv}(\overline X)$,
but their maximum is the
point $(6,4,3/2)$ which is not in ${\rm conv}(\overline X)$. In fact, 
$$ {\rm MMconv}(X) \,\, =\,\, {\rm conv}\biggl(X\cup \bigl\{ \,\bigl( \,6,4,\frac32 \,\bigr)\,  \big\} \biggr). $$
The bimonotone inequalities that describe $\MMconv(X)$ are 
$$z \geq 0, \quad y-2z\geq 0,\quad x-4z\geq 0,\quad 2x-3y\geq 0, \quad y-4\leq 0,\quad x-z\leq 6.$$
These halfspaces furnish the description of MMconv$(X)$  in Theorem~\ref{2DProjectionsTheorem}.
This polytope is a bipyramid,  discussed  in more detail in Example \ref{ex:3d_main}.
 \qed
\end{example}

We now turn to log-concave LLC distributions. Here the support of the MLE is generally larger than ${\rm MMconv}(X)$. In the proof of Theorem~\ref{thm_ex_unique}
we show that, if an MLE exists, its domain is given by the $L^\natural$-extension~of~$X$.

\begin{definition}\label{defn:l.extension}
Let $X \subset \R^d$ be a finite set. We say that $X$ is \emph{$L^\natural$-closed} if there exists some $r > 0$ such that 
\begin{equation}\label{eqn:llc.closed}
x,y \in X \,\,\,\text{implies}\,\,\, (x + \alpha \cdot r\mathbf{1}) \wedge y,\, x \vee (y- \alpha \cdot r\mathbf{1}) \in X
\quad \text{for all $\alpha \in \mathbb{Z}_+$.}
\end{equation} 
 We say that $X$ is \emph{$L^\natural$-convex} if 
$$
\quad x,y \in X \,\,\,\text{implies}\,\,\, (x + \alpha \cdot \mathbf{1}) \wedge y, \, x \vee (y-  \alpha \cdot \mathbf{1}) \in X 
\quad \text{for all $\alpha \in \mathbb{R}_+$.}
$$
The \emph{discrete $L^\natural$-extension} of $X$, denoted $\tilde{X}$, is the smallest finite 
subset of $\RR^d$ that is $L^\natural$-closed  and contains
$X$. The (continuous) \emph{$L^\natural$-extension} of $X$, denoted $L(X)$, is the smallest $L^\natural$-convex set  in $\RR^d$ containing $X$.
\end{definition}

Unlike the $\mtp$ case, a finite set $X$ may not have a discrete $L^\natural$-extension. However, it always has a continuous $L^\natural$-extension $L(X)$. This is a polytope in $\RR^d$.
The following proposition gives the inequality description for $L(X)$,
and it characterizes when $\tilde{X}$ exists. The proof is given in Appendix~\ref{app_support}.

\begin{proposition}\label{prop:llc-extension}
Let $X \subset \R^d$ be a finite set. Then its $L^\natural$-extension is
\begin{equation}\label{eqn:px}
L(X) \,=\, \bigl\{y \in \mathbb{R}^d: y_i - y_j \leq \max_{z \in X} (z_i - z_j), 
\,\,\, \min_{z \in X} z_i \leq y_i \leq \max_{z \in X} z_i \bigr\}. 
\end{equation}
The set $X$ admits a discrete $L^\natural$-extension $\tilde{X}$ if and only if 
$r(X-v) \subset \mathbb{Z}^d$ for some $r > 0$ and $v \in X$.
In this case, there is a unique smallest constant $r^\ast > 0$, independent of the choice of $v$, such that 
$r(X-v) \subset \mathbb{Z}^d$, and
\begin{equation}\label{eqn:tilde.x}
\tilde{X} \,\,=\,\, v+1/r^\ast\cdot \biggl(L(r^\ast(X-v)) \cap \mathbb{Z}^d \biggr).
\end{equation}
If $X$ has a discrete $L^\natural$-extension $\tilde{X}$, then $L(X) = L(\tilde{X}) = \conv(\tilde{X})$. 
\end{proposition}

\begin{figure}[h]
\centering{\includegraphics[width=0.509\textwidth]{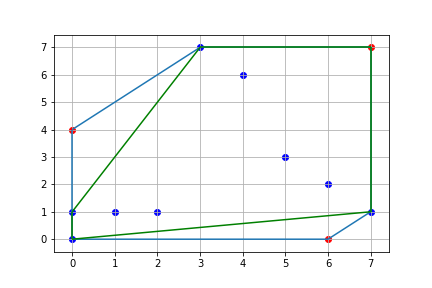}}
\caption{\normalsize An example of a set $X$ (blue points), its $L^\natural$-extension $L(X)$ (the outer hexagon), and its min-max convex hull $\MMconv(X)$ (the inner hexagon). Note that both $L(X)$ and $\MMconv(X)$ contains extra added points (red points).}
\end{figure}

Any finite set of rational points admits a discrete $L^\natural$-extension; so for practical purposes, $\tilde{X}$ always exists. In that case $L(X) = L(\tilde{X}) = \conv(\tilde{X})$, a key difference compared to the general $\mtp$ setting. Computation of $\tilde{X}$ and $L(X)$ is immediate from their definitions. We summarize the main results of this section in the following proposition, which is proven in Appendix~\ref{app_support}.

\begin{proposition}
\label{prop_support} Fix a sample
 $X = \{x_1,\ldots, x_n\}$ in $\mathbb{R}^d$.
If the log-concave $\mtp$ MLE exists, then its support equals the min-max convex hull ${\rm MMconv}(X)$. 
If the LLC MLE exists, then its support equals $L(X)$.
\end{proposition}

\section{The MLE for tidy configurations}
\label{sec_3}

Fix $X = \{x_1,\ldots, x_n\}\subset\mathbb R^d$.  For a vector $y = (y_1,\ldots, y_n) \in \RR^n$, the \emph{tent function} $h_{X, y}:\mathbb R^d \to\mathbb R\cup\{\-\infty\}$ is the smallest  concave function satisfying $h_{X,y}(x_i) \geq y_i$ for 
all~$i$. This function is piecewise linear, it is greater
 than~$-\infty$ on $\conv(X)$, and equals $-\infty$ outside $\conv(X)$. 
In this section, we provide sufficient conditions for the MLE under log-concavity and  $\mtp$/LLC to equal the exponential of a tent function. 
We show that this is the case when $d=2$ or when $X\subseteq \prod_{i=1}^d\{a_i, b_i\}$ for some $a_i\leq b_i$, $i=1,\ldots, d$, under $\mtp$, and when $X\subset \mathbb{Q}^d$ under LLC. 
If we know that the MLE equals the exponential of a tent function
then we only need to find the heights $y\in\mathbb R^n$. This is a finite-dimensional problem.

\begin{definition}\label{def:supermodular_heights} Given $X$, a vector of heights $y\in\mathbb R^n$ is {\em supermodular} if 
$$y_i+y_j \leq y_k + y_\ell \quad\text{ whenever }\quad x_k = x_i\wedge x_j \, \text{ and }\, x_\ell = x_i\vee x_j.$$
The vector $y\in\mathbb R^n$ is {\em $L^\natural$-concave} if, for all $\alpha > 0$,
\vspace{-0.2cm}
$$y_i+y_j\leq y_k+y_\ell \,\, \text{ whenever }\,\, x_k=(x_i+\alpha\mathbf{1})\wedge x_j\, \text{ and }\, x_\ell = x_i\vee(x_j-\alpha\mathbf{1}).$$
The vector $y\in\mathbb R^n$ is \emph{tight} or {\em concave} if $h_{X,y}(x_i) = y_i$ for all $i \in \{1,\ldots,n\}$.
\end{definition}
In the following, we show that the MLE is a tent function for special configurations $X$ for which supermodularity/$L^\natural$-concavity of the heights $y$ already implies that the tent function $h_{X,y}$ is supermodular/$L^\natural$-concave.

\begin{definition} \rm
Let $X = \{x_1,x_2,\ldots, x_n\}$ be a subset of $\mathbb{R}^d$. We say that $X$ is {\em MTP$_2$-tidy} if, for any tight vector of supermodular heights $y\in\mathbb R^n$, the tent function $h_{X,y}$ is supermodular. We say that $X$ is {\em LLC-tidy} if, for any set of tight $L^\natural$-concave heights $y \in \mathbb{R}^n$, the tent function $h_{X,y}$ is $L^\natural$-concave. 
\end{definition}

Note that $\mtp$-tidy/LLC-tidy implies that $X$ is min-max closed/equals its discrete $L^\natural$-extension. We now state our main result of this section. It is, in general, nontrivial to check whether a set $X$ is tidy. In Theorem~\ref{thm:tidy} we describe several classes of such sets.

\begin{theorem}\label{thm:tidy_solution} If $X$ is $\mtp$-tidy/LLC-tidy, then the optimal solution to the  MLE problem is the exponential of a tent function on $X$.
The tent pole heights $y$ can be found by solving the finite-dimensional convex program
	\begin{equation}
	\label{tidy_opt_problem}
	\text{minimize}\quad - w\cdot y + \int_{\mathbb R^d} exp(h_{X,y}(z))dz
	\quad \quad	\text {subject to}\quad y\in\mathcal S,
	\end{equation}
	where $\mathcal S$ is the set of supermodular heights $y$ on $X$ for the $\mtp$ case, and it is the set of $L^\natural$-concave heights on $X$ for the LLC case.
\end{theorem}

\begin{proof}
We here prove the $\mtp$ case.  The LLC case is the same, with `supermodular' replaced by `$L^\natural$-concave'.
Suppose $X$ is $\mtp$-tidy, and let $\mathcal{S}$ be the set of supermodular heights on $X$. Note that 
 $\mathcal S$ is a convex set. Indeed, $\mathcal S$ is defined by the linear inequalities from Definition~\ref{def:supermodular_heights}. Using Lagrange multipliers, one can see that (\ref{tidy_opt_problem}) is equivalent~to 
	\begin{equation}
	\label{tidy_opt_problem.2}
	\text{minimize}\quad - w\cdot y
\quad\,\,\, 	\text {subject to} \quad y\in\mathcal S
\,\,\, {\rm and} \,\, \int_{\mathbb R^d} exp(h_{X,y}(z))dz = 1. 
	\end{equation}
	 Any optimizer $y$ of (\ref{tidy_opt_problem.2}) gives rise to a feasible solution $f = \exp(h_{X,y})$ to (\ref{eq:optproblem}), because
	  (\ref{tidy_opt_problem.2}) minimizes the same objective function over the subset of log-piecewise linear and concave
	  $\mtp$ densities on $X$.

Now, let $f$ be any feasible solution to (\ref{eq:optproblem}). Define $y'_i = {\rm log}(f(x_i))$ for $i=1,2,\ldots,n$. 
We shall exhibit a feasible solution $y$ of (\ref{tidy_opt_problem.2}) that
satisfies $-w \cdot y \leq - w \cdot y'$. This will prove the claim. 
We abbreviate  $g = {\rm exp}(h_{X,y'})$. As $h_{X,y'}$ is defined to be the smallest concave function with the
	given values on $X$, we have $g \leq f$ pointwise, and hence
	$\int_{\RR^d} g(x) dx \leq 1$.
    Let $c = 1/\int_{\RR^d} g(x) dx$ and define $y$ by setting $y_i := y'_i + \log c$
    for all $i$. Since $X$ is tidy and $f$ is $\mtp$, 
    we have $y \in \mathcal{S}$, and since $\log f$ is concave, then $y$ is tight. 
    Furthermore, $\int \exp(h_{X,y}(z)) \, dz = c\int \exp(h_{X,y'}(z)) \, dz = 1$. 
    Therefore, $y$ is a feasible solution of (\ref{tidy_opt_problem.2}), and $ -w \cdot y \leq - w \cdot y'$, as desired. 
\end{proof}

It is hence of interest to characterize the configurations that are $\mtp$-tidy or LLC-tidy. In the $\mtp$ case, a necessary condition is that $X$ is min-max closed, that is, $X = \overline{X}$. We show that in dimension two or when $X$ is binary, this is also sufficient. The proof is provided in Appendix~\ref{app_3}.

\begin{theorem} 
	\label{thm:tidy}
	Let $X = \overline X$ be finite min-max closed subset of $\RR^d $.
	If $d=2$  or if $X\subseteq \prod_{i=1}^d\{a_i,b_i\}$ where $a_i \leq b_i$ for all $i$,
	  then $X$ is MTP$_2$-tidy.
\end{theorem}

Outside the cases covered by Theorem~\ref{thm:tidy},
there is an abundance of min-max closed configurations $X$ that fail to be MTP$_2$-tidy. This happens even if ${\rm conv}(X)$ is min-max closed, as in Example~\ref{ex_chain}. We conjecture that two-dimensional and binary are essentially the only $\mtp$-tidy configurations.

\begin{conjecture}\label{conj_tidy} 
A min-max closed configuration $X$ of points in $\mathbb R^d$ for which $\conv(X) =$ MMconv$(X)$ is $\mtp$-tidy if and only if $d = 2$ or $X\subseteq\prod_{i=1}^d\{a_i, b_i\}$, where $a_i\leq b_i\in\mathbb R$ for $i=1,\ldots, d$.
\end{conjecture}

In other words,  if $d\geq 3$ and three points $x_1, x_2, x_3$ in $X$ have distinct $i$-th coordinates,
for some $i$, then we conjecture that the configuration $X$ is not tidy. 
Therefore, when $d\geq 3$ and the number of samples is $n\geq 3$, we conjecture that the configuration will not be MTP$_2$-tidy with probability one. Nevertheless, we still believe that in those cases the MLE is the exponential of a tent function. See Section~\ref{sec:non-tidy} for more details.
The situation is much better for LLC-tidiness. Similar to the $\mtp$ case, a necessary condition for $X$ to be LLC-tidy is that it admits a discrete $L^\natural$-extension. The following states that this condition is also sufficient.

\begin{theorem}[Theorem 7.26 in \cite{	murota2003discrete}] \label{thm:WeLoveRationals}
Let $X$ be a finite configuration in $\RR^d$ that
admits a discrete $L^\natural$-extension $\tilde{X}$.
Then $\tilde{X}$ is LLC-tidy. In particular, if $X \subset \mathbb{Q}^d$, then its discrete $L^\natural$-extension is LLC-tidy.
\end{theorem}

For computational purposes, numbers in $\mathbb{R}$ are usually rounded to numbers in $\mathbb{Q}$.
Theorem \ref{thm:WeLoveRationals} hence means that, for practical purposes, any
given sample $X$ can be extended to an LLC-tidy configuration.
Theorem~\ref{thm:tidy_solution} implies that the LLC MLE is the exponential of a tent function
which can be computed by solving a finite-dimensional convex optimization problem. 

Note that $\mtp$-tidiness does not imply LLC-tidiness: by Theorem \ref{thm:tidy}, $X \subset \R^2$ is $\mtp$-tidy, but not every finite set of points in $\R^2$ can be extended to a finite LLC-tidy configuration by Proposition \ref{prop:llc-extension}. Neither does being LLC-tidy imply $\mtp$-tidy; see Example \ref{ex:twelvepoints}. One instance that is both $\mtp$-tidy and LLC-tidy is the cube
 $X = \{0,1\}^d$. Here, the tight supermodular function $h_{X,y}$ is known as the \emph{Lov\'{a}sz extension} of $y$ \cite{lovasz1983submodular}. 

\section[Algorithms]{Algorithms for computing the MLE for tidy configurations}
\label{sec_4}

As shown in Theorem~\ref{thm:tidy_solution}, for tidy configurations, the optimization problem~\eqref{tidy_opt_problem} finds the MLE in both the $\mtp$ and LLC cases. 
The log-concave MLE problem amounts to maximizing the objective function 
in~\eqref{tidy_opt_problem} over $\mathbb R^n$. 
The objective function is not everywhere differentiable, and the authors of~\cite{CSS} 
use subgradients to maximize it. Recently, various papers have proposed 
 accelerating the computation of the log-concave MLE from exponential to
 polynomial time by efficiently computing (sub-)gradients of (an approximation of) the 
 likelihood function~\cite{AxeVal18,DiaSidSte18,RatSch18}. 
 We are optimistic that these approaches can be extended to
obtain a polynomial time algorithm also for 
 computing  the log-concave MTP$_2$ MLE.

Our implementation uses the original exponential time algorithm in~\cite{CSS} for computing subgradients
combined with the \emph{Frank-Wolfe}  (or \emph{conditional gradient}) method for solving the constrained problem~\eqref{tidy_opt_problem}. Even though we use subgradients instead of gradients at points of non-differentiability, the Frank-Wolfe algorithm is guaranteed to find the global optimum~\cite{W}.

Algorithm~\ref{alg:tidy_estimation} computes
 the $\mtp$ MLE for configurations $X$ in $\mathbb R^2$.
 It easily extends to other tidy configurations.
Its first step is Algorithm~\ref{alg:min-max_closure} for computing the min-max closure $\overline X$ of the set $X$. The algorithm uses
an alternative characterization of MMconv$(X)$ (cf. Lemma \ref{lem:mmconv} in Appendix~\ref{app_4}).
As a  speedup, our algorithm iterates through the points $(a_1,b_2)$ in $\RR^2$,
for all pairs $a,b\in X$, and checks whether they lie in MMconv$(X)$. The resulting set $X'$ gives the
 tent pole locations. Then the algorithm computes the set $\mathcal S$ of tight supermodular heights, as described in Algorithm~\ref{alg:supermodular_heights}, and uses the conditional gradient method 
 (Algorithm~\ref{alg:FrankWolfe_sub}) to find the optimal solution. 

The LLC analogue is given in Algorithm \ref{alg:tidy_estimation.llc}. It computes the LLC MLE for rational 
configurations $X\subset \mathbb Q^d$.  The $L^\natural$-extension of a set $X$ is computed using Algorithm \ref{alg:extension.llc}. Then
Algorithm \ref{alg:supermodular_heights.2} computes $\mathcal{S}$ based on an alternative characterization of $L^\natural$-concavity given in \cite[Proposition 7.5]{murota2003discrete}. The proof of  correctness is given in Appendix~\ref{app_4}. Finally, the Frank-Wolfe method (Algorithm~\ref{alg:FrankWolfe_sub}) solves the optimization problem using subgradient optimization. 

We now illustrate how Algorithms~\ref{alg:tidy_estimation} and \ref{alg:tidy_estimation.llc} perform on 
Gaussian samples.

\begin{figure}[!b]
	\vspace{-0.5cm}
	\includegraphics[width=0.3\textwidth]{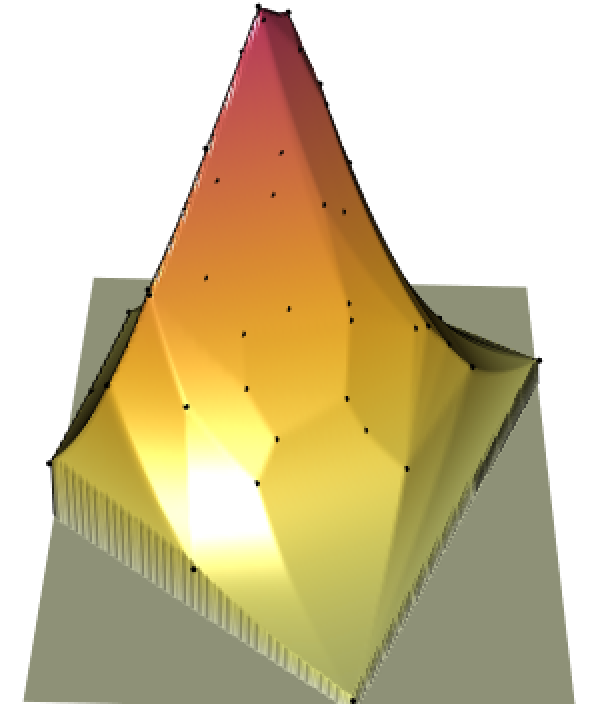}\hspace{0.25cm}
	\vspace{-0.3cm}\includegraphics[width=0.33\textwidth]{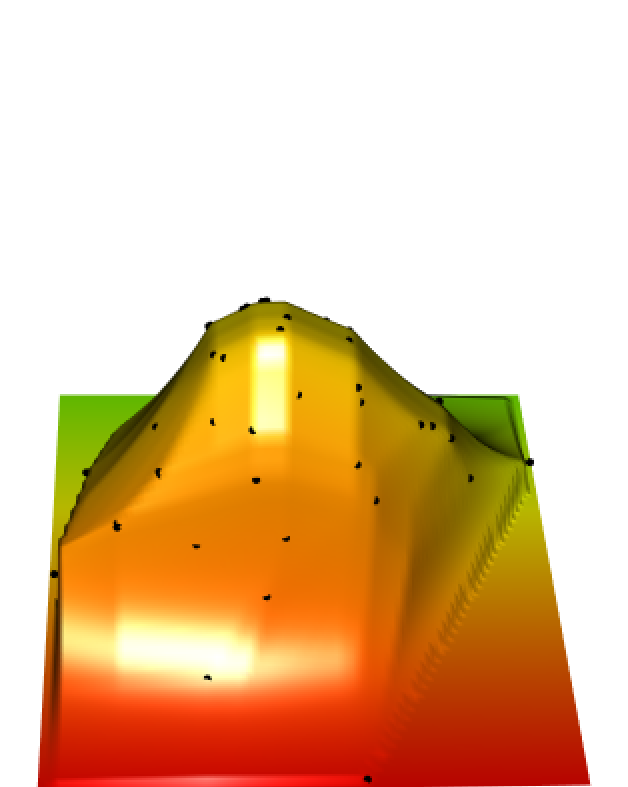}\hspace{0.25cm}
	\includegraphics[width=0.3132\textwidth]{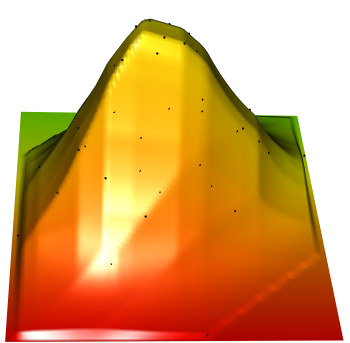}
	\caption{Density estimates for 55 samples from a standard Gaussian distribution in $\mathbb{R}^2$; Log-concave using~\cite{CGS} (left), log-concave MTP$_2$ (middle), log-concave LLC (right).}
	\label{2d_estimation}
\end{figure}

	\begin{algorithm}[!t]
		\label{alg:tidy_estimation}
		\caption{Computing the $\mtp$ MLE in $\mathbb{R}^2$}
		\LinesNumbered
		\DontPrintSemicolon
		\SetAlgoLined
		\SetKwInOut{Input}{Input}
		\SetKwInOut{Output}{Output}
		\Input{$n$ samples $X\subset\mathbb R^2$ with weights $w\in\mathbb R^n$.}
		\Output{the optimal heights $\overline y\in\mathbb R^N$, $N = |\overline X|$, corresponding to the points in $\overline X$.}
		\BlankLine
		Compute the min-max closure $\overline X$ of $X$ using Algorithm~\ref{alg:min-max_closure};
		\;
		Set the new weights $\overline w$ equal to the old weights $w$ padded with zeros for each of the extra points in $\overline X$;
		\;
		Compute the set $\mathcal S$ of inequalities for all supermodular heights $\overline y$ using Algorithm~\ref{alg:supermodular_heights};
		\;
		Compute $\overline y$ as the optimum for \eqref{tidy_opt_problem} using Algorithm~\ref{alg:FrankWolfe_sub}.
	\end{algorithm}
	\begin{algorithm}[!t]
		\label{alg:tidy_estimation.llc}
		\caption{Computing the LLC MLE for $X \subset \mathbb{Q}^d$}
		\LinesNumbered
		\DontPrintSemicolon
		\SetAlgoLined
		\SetKwInOut{Input}{Input}
		\SetKwInOut{Output}{Output}
		\Input{$n$ samples $X\subset\mathbb Q^d$ with weights $w\in\mathbb R^n$.}
		\Output{the optimal heights $\tilde y \in \mathbb{R}^N$, $N = |\tilde X|$, corresponding to the points in $\tilde X$.}
		\BlankLine
		Compute the discrete $L^\natural$-extension $\tilde X$ of $X$ using Algorithm \ref{alg:extension.llc}		\;
		Set the new weights $\tilde w$ equal to the old weights $w$ padded with zeros for each of the extra points in $\tilde X$;
		\;
		Compute the set $\mathcal S$ of inequalities for all L$^{\#}$-concave heights $\tilde y$ using Algorithm~\ref{alg:supermodular_heights.2};
		\;
		Compute $\tilde y$ as the optimum for \eqref{tidy_opt_problem} using Algorithm~\ref{alg:FrankWolfe_sub}.
	\end{algorithm}
	\begin{algorithm}[!h]
		\label{alg:min-max_closure}
		\caption{Computing the min-max closure of a finite set $X\subset\mathbb{R}^2$}
		\LinesNumbered
		\DontPrintSemicolon
		\SetAlgoLined
		\SetKwInOut{Input}{Input}
		\SetKwInOut{Output}{Output}
		\Input{a finite set of points $X\subset\mathbb R^2$;}
		\Output{the min-max closure $\overline X$ of $X$.}
		\BlankLine
		Create an  $n_1\times n_2$ index table $T$ filled with zeros corresponding to grid$(X)$, where $n_i$ is the number of different $i$-th coordinates among the points in $X$.
		\;
		For each $x_i \in X$, find its place in $T$ and set the corresponding value of $T$ to $i$.
		\;
		Set $\overline X = X$;
		\;
		Let $C = $ conv$(X\cup \{\min(X), \max(X)\})$ and counter $= \#X + 1$;
		\;
		For $i$ from 1 to $n_1$ do:\;
		\hspace{1em} For $j$ from 1 to $n_2$ do:\;
		\hspace{2em} If the point $p[i,j]$ corresponding to position $i,j$ in $T$ is in $C$ then set \;
        \hspace{3em} $T[i,j] =$ counter; counter $=$ counter $+ 1$; Add $p[i,j]$ to $\overline X$;
		\;
		Output $\overline X$.
	\end{algorithm}

{
	\begin{algorithm}[t]
		\label{alg:extension.llc}
		\caption{The~discrete~$L^\natural$-extension~of~a~finite~set~$X\!\subset\!\mathbb{Q}^d$}
		\LinesNumbered
		\DontPrintSemicolon
		\SetAlgoLined
		\SetKwInOut{Input}{Input}
		\SetKwInOut{Output}{Output}
		\Input{a finite set of points $X\subset\mathbb{Q}^d$;}
		\Output{the discrete $L^\natural$-extension $\tilde{X}$ of $X$.}
		\BlankLine
		Let $m \in \mathbb{N}$ be the smallest integer such that $m\cdot X \subset \mathbb{Z}^d$;\;
        Compute $L(m\cdot X)$ via \eqref{eqn:px}; \;
        Set $X' = L(m\cdot X) \cap \Z^d$; \;
        Output $\tilde{X} = \{\frac{1}{m}x': x' \in X'\}$.
	\end{algorithm}
	\begin{algorithm}[!h]
		\label{alg:supermodular_heights}
		\caption{Computing the inequalities $\mathcal{S}_{\textrm{supermodular}}$ in $\mathbb{R}^2$}
		\LinesNumbered
		\DontPrintSemicolon
		\SetAlgoLined
		\SetKwInOut{Input}{Input}
		\SetKwInOut{Output}{Output}
		\Input{the index table $T$ generated for $X$ in Algorithm~\ref{alg:min-max_closure}}
		\Output{the set of inequalities defining the cone of supermodular heights $\mathcal S$}
		\BlankLine
        Set $\mathcal S$ to be an empty set of inequalities \;
		For $i$ from $1$ to $n_1-1$:\;
		\hspace{1em} For $j$ from 1 to $n_2-1$:\;
		\hspace{2em} If $T[i,j]\neq0, T[i+1,j]\neq 0, T[i,j+1]\neq 0, T[i+1,j+1]\neq 0$, then \; 
        \hspace{3em} add the inequality $y_{T[i,j]} + y_{T[i+1,j+1]} - y_{T[i+1,j]} - y_{T[i,j+1]} \geq 0$ to $\mathcal S$;\;
		Output $\mathcal S$.
	\end{algorithm}
	\begin{algorithm}[!h]
		\label{alg:supermodular_heights.2}
		\caption{Computing the inequalities $\mathcal S_{L^\natural\textrm{-concave}}$ in $\mathbb{Z}^d$}
		\LinesNumbered
		\DontPrintSemicolon
		\SetAlgoLined
		\SetKwInOut{Input}{Input}
		\SetKwInOut{Output}{Output}
		\Input{Finitely many points $X \subset \mathbb{Z}^d$ that are LLC-tidy}
		\Output{the set of inequalities defining the cone of $L^\natural$-concave heights $\mathcal S$}
		\BlankLine
        Let $e_i = $ $i$-th coordinate vector, $e_0 = -\sum_{i=1}^d e_i$,\, $\mathcal S$ empty set of inequalities; \;
		For each $x \in X$: \;
        \hspace{1em} For each $i,j = 0, 1, \dots, d$, $i < j$: \;
        \hspace{2em}  	If $x + e_i + e_j, x + e_i, x+e_j \in X$ then\; 
        \hspace{3em} add the inequality
                	\,$y(x) + y(x+e_i+e_j) - y(x+e_i) - y(x+e_j) \geq 0$\,  to $\mathcal S$\;
		Output $\mathcal S$.
	\end{algorithm}
	\begin{algorithm}[!h]
		\label{alg:FrankWolfe_sub}
		\caption{Frank-Wolfe update}
		\LinesNumbered
		\DontPrintSemicolon
		\SetAlgoLined
		\SetKwInOut{Input}{Input}
		\SetKwInOut{Output}{Output}
		\Input{objective function $f:\mathbb R^n\to\mathbb R$, function $g: \mathbb R^n\to\mathbb R^n$ producing a subgradient of $f$, feasible set $\mathcal S\subseteq\mathbb R^n$.}
		\Output{solution $y^* = \text{argmin}_{y\in\mathcal S} f(y)$.}
		\BlankLine
		Initialize $y^{(0)}$ in $\mathcal S$
		\;
		For $k$ from 1 to $N$ do:
		\;
		\hspace{1em} Set $s^{k-1} = $ argmin$_{s\in\mathcal S}\,\, g(y^{(k-1)})^T\,\, s$ via a linear program;
		\;
		\hspace{1em} Set $y^{(k)} = (1 - \gamma_k)y^{(k-1)} + \gamma_ks^{(k-1)}$, where $\gamma_k = \frac2{k+1}$;
		\;
        Output $y^{(N)}$.
	\end{algorithm}

\begin{example}
Let $X$ consist of 55 i.i.d.~samples from a standard Gaussian distribution  in $\mathbb R^2$. In Figure~\ref{2d_estimation}, we show the corresponding log-concave density estimator on the left, the log-concave MTP$_2$ density estimator in the middle, and the log-concave LLC density estimator on the right. The log-concave density estimator was computed using the package {\tt LogConcDEAD} described in~\cite{CGS}. The log-concave MTP$_2$ density estimator was computed using Algorithm~\ref{alg:tidy_estimation} for solving the optimization problem~\eqref{tidy_opt_problem}. The min-max closure $\overline X$ was computed using Algorithm~\ref{alg:min-max_closure} and consisted of the original 55 samples plus 2691 additional points. The LLC density estimator was computed using Algorithm \ref{alg:tidy_estimation.llc} with input $X'$ obtained by rounding each point in $X$ up to the first decimal place. The discrete $L^\natural$-extension of $X'$ was computed using Algorithm~\ref{alg:extension.llc} and consisted of the original 55 samples plus 2496 additional points. An implementation is provided in
\begin{center}\verb+http://github.com/erobeva/MTP_2_density_estimation+ .\end{center}
\noindent 
The $\ell_2$, Hellinger, and $\ell_\infty$ distances between the estimates and the true distribution in our small example are shown in Table~1. 
This example indicates the power of MTP$_2$ and LLC constraints: with only 55 samples, the log-concave MTP$_2$ and LLC estimators look close to a Gaussian distribution, while the log-concave density estimator is still quite rough. 

We leave the computation of the statistical rates of convergence of our estimators to future work. We believe that the artificial addition of all the points necessary to obtain the min-max closure and the $L^{\natural}$-extension is potentially responsible for a faster statistical rate of convergence, though they do form a computational expense. This number of points appears to be on the order of $n^d$ on average, which leads to $O(n^d)$ "samples" that we use to compute our estimators. Note that while LLC is a much smaller subset and has nice theoretical guarantees, finding the exact LLC MLE in the non-tidy case is more computationally expensive than finding the approximate MLE under the MTP$_2$ constraint. This is beause the $L^\natural$-extension of $X$ is generally larger than MMconv$(X)$, and its discretization certainly contains more points than those used by the optimization algorithm for finding the MLE under MTP$_2$. Thus, a practical solution would be to follow our suggested algorithm to compute the MTP$_2$ MLE, and if it is not LLC, then one can compute the exact LLC estimator. This brings forth the interesting open question of how closely the LLC estimator approximates the MTP$_2$ estimator under model misspecifications, which we leave for future work.

\begin{center}
\begin{table}
 \hspace{2cm} \begin{tabular}{| c |c c c |}
    \hline
    &log-concave&log-concave $\mtp$& log-concave LLC \\ \hline
    $\ell_2$-loss & 0.0115093 & 0.00539397 & 0.00627162 \\
    Hellinger loss & 0.1210729 & 0.09313234 & 0.06755705 \\
    $\ell_\infty$-loss & 0.0981327 & 0.03811975 & 0.04853513 \\ \hline
  \end{tabular}
    \caption{Comparison of the errors of the log-concave, log-concave \& MTP$_2$, and log-concave \& $L^{\natural}$-concave estimators.}
    \end{table}
\end{center}
\end{example}

\section{The MLE in the general case}
\label{sec_5}

By Theorem~\ref{thm:tidy_solution}, in the tidy case, the MLE is always the exponential 
of a tent function and can be found by solving a finite-dimensional convex
optimization problem.
In this section we explore the non-tidy case. 
A na\"{i}ve approach is to start by
 computing the log-concave MLE.
 Following \cite{CSS},
  this is the exponential of a tent function. If that MLE is $\mtp$/LLC, then it is also the 
  MLE under log-concavity plus $\mtp$/LLC. It is therefore of interest to 
  characterize supermodular/$L^\natural$-concave tent functions. 
 Although we do not prove that the MLE is always the exponential of a tent function 
in the non-tidy case, we provide partial justification towards this conjecture. At the end of this section we compute the MLE in a specific non-tidy example.

\subsection{Characterization of supermodular/$L^\natural$-concave tent functions}

Without loss of generality we assume that our sample $X$ satisfies $X = \overline X$ and 
MMconv$(X) = $ conv$(X)$. If this is not the case, Algorithm~\ref{alg:QT} can be used to add points to $X$ so that those two conditions are satisfied. Every piecewise linear function $h_{X,y}$ induces a regular 
polyhedral subdivision of the point configuration $X$. A {\em polyhedral subdivision of a point configuration $X$} is a polyhedral complex $\Delta\subseteq\RR^d$ whose set of vertices is a subset of $X$ and the union of whose cells, denoted $|\Delta|$, equals conv$(X)$. The cells in $\Delta$ are themselves convex polyhedra of various dimensions, and any two cells
intersect in a polyhedron that is a face of each. If all cells in $\Delta$ are simplices
then $\Delta$ is called a {\em triangulation}. A polyhedral subdivision $\Delta$ of $X$ is {\em regular} if there exists a set of heights $y\in\RR^n$ such that the polyhedra in $\Delta$ are the regions of linearity of the piecewise linear function $h_{X,y}$.
We refer to the textbook \cite{DRS} for all relevant basics on subdivisions and triangulations.

We are now ready for the main result in this section which is the following characterization, in terms of polyhedral subdivisions, of the functions $\,f : \RR^d \rightarrow \RR \cup \{-\infty\}\,$ that are
 concave, supermodular, and piecewise linear.
 
\begin{theorem}\label{thm:BimonotoneSubdivision}
 Let $f:\mathbb R^d \to \mathbb R \cup \{-\infty\}$ be a piecewise linear concave function. Let $\Delta$ be the subdivision of the effective domain of $f$, whose cells are the regions of linearity of~$f$. Then, $f$ is supermodular if and only if each 
 cell of $\Delta$ is min-max closed, i.e.~defined by bimonotone linear inequalities.
\end{theorem}

The proof of Theorem~\ref{thm:BimonotoneSubdivision} is given in Appendix~\ref{app:thmBimonotone}. The analogous theorem for $L^\natural$-concave tent functions is the following result of Murota.

\begin{theorem}[Theorem 7.45, \cite{murota2003discrete}]
 Let $f:\mathbb R^d \to \mathbb R \cup \{-\infty\}$ be~concave and
 piecewise linear, and $\Delta$ the 
 induced subdivision of the effective domain of $f$.
   Then, $f$ is $L^\natural$-concave if and only if each cell of $\Delta$ is $L^\natural$-concave.
\end{theorem}

Theorem~\ref{thm:BimonotoneSubdivision} can be used to 
give configurations that are not $\mtp$-tidy.

\begin{example}
\label{ex_chain}
\rm
Let $ X = \{(0,0,0),  (6,0,0),  (6,4,0), (6,4,\frac{3}{2}), (8,4,2)\}$ in $\RR^3$.
This is a chain in $(\mathbb{Z}^3, \leq)$. Thus,
$X$ is min-max closed, and any height vector $y\in\RR^5$ is supermodular.
The configuration $X$ has precisely two triangulations. Neither of them is
bimonotone. Hence $X$ is not $\mtp$-tidy. 
\end{example}

Next, we study the set $\mathcal S \subset \RR^n$ of height vectors $y$ which induce a supermodular tent function. In the tidy setting, $\mathcal S$ is a convex set, since it is the set of all tight supermodular heights. However, in the non-tidy setting, $\mathcal S$ is often not convex, as illustrated in the following example.

  \begin{example}[Two cubes in $\mathbb R^3$]\label{ex:twelvepoints} 
  Fix the $n=12$ points in $\RR^3$ given~by
  \setcounter{MaxMatrixCols}{20}
 $$ X \,\, = \,\,
 \begin{pmatrix} 
  0 & 0 & 0 & 0 & 0 & 0 & 1 & 1 & 1 & 1 & 1 & 1 \\
  0 & 0 & 0 & 1 & 1 & 1 & 0 & 0 & 0 & 1 & 1 & 1 \\
  0 & 1 & 2 &  0 & 1 & 2 &   0 & 1 & 2 & 0 & 1 & 2
  \end{pmatrix}.  $$
Thus $X$ consists of the lattice points in a $  1 \times 1 \times 2$ box, obtained by adjoining two $3$-cubes. 
Consider the height vectors $y =  (5, 6, 1, 1, 4, 1, 1, 4, 1, 1, 6, 5)$ and
$y' = (11, 9, 1, 1, 0, 0, 10, 12, $ $ 16, 11, 21, 27)$.
These induce two distinct regular triangulations $\Delta$ and $\Delta'$ of $X$.
In terms of column labels for $X$, they are
$$ \begin{small} \begin{matrix} \! \Delta\, =\,
\bigl\{ \{1, 2, 5, 11\}, \{1, 2, 8, 11\}, \{1, 4, 5, 11\}, \{1, 4, 10, 11\}, \{1, 7, 8, 11\},\{1, 7, 10, 11\},  \\  \qquad \,\quad\,
\{2, 3, 6, 12\}, \{2, 3, 9, 12\}, \{2, 5, 6, 12\}, \{2, 5, 11, 12\}, \{2, 8, 9, 12\}, \{2, 8, 11, 12 \}  \bigr\}, \end{matrix} 
\end{small} $$
$$ \begin{small} \begin{matrix} \Delta' & = &    \bigl\{
 \{1, 2, 6, 12\}, \{1, 2, 9, 12\}, \{1, 4, 6, 12\}, \{1, 4, 10, 11\},  \{1, 4, 11, 12\}, \\ & &  \quad \,
  \{1, 7, 9, 12\}, \{1, 7, 10, 11 \}, \{1, 7, 11, 12 \}, \{ 2, 3, 6, 12 \}, \{2, 3, 9, 12 \} \bigr\}. \end{matrix}
  \end{small}
$$

A computation reveals that a subdivision of $X$ is bimonotone if and only if  each of its walls 
is spanned by points indexed by subsets in
$$ \begin{small} \begin{matrix} 
 \{1, 4, 7, 10\},
 \{1, 4, 8, 11\},
 \{1, 4, 9, 12\},
 \{1, 5, 7, 11\},
 \{1, 6, 7, 12\},
 \{2, 5, 8, 11\}, \\
 \{2, 5, 9, 12\},
 \{2, 6, 8, 12\}, 
 \{3, 6, 9, 12\},
 \{1, 2, 3, 10, \! 11, \!12\},
 \{1, 2, 3, 4, 5, 6\},\\
 \{1, 2, 3, 7, 8, 9 \},
 \{4, 5, 6, 10, \! 11, \!12\},
 \{7, 8, 9, 10, \! 11,\! 12 \}.
 \end{matrix} \end{small}
 $$
Hence, both $\Delta$ and $\Delta'$ are bimonotone, so $y$ and $y'$ induce supermodular tent functions. However, the convex combination  $y'' =  \frac{5}{12} y + \frac{7}{12} y'$ induces 
$$ \begin{small} \begin{matrix} \Delta''  =
  \bigl\{ \{ {\bf 1},  2, {\bf 5}, {\bf 12}\}, \{1, 2, 9, 12\}, \{{\bf 1}, 4, {\bf 5}, {\bf 12} \}, \{1, 4, 10, 11\}, \{1, 4, 11, 12\}, 
     \{1, 7, 8, 12\},   \\   \qquad \quad
\{1, 7, 10, 11\}, \{1, 7, 11, 12\}, \{1, 8, 9, 12\}, \{2, 3, 6, 12\}, \{2, 3, 9, 12\}, \{2, 5, 6, 12 \} \bigr\}. \end{matrix}
\end{small} $$
The regular triangulation  $\Delta''$ is not bimonotone because the wall given by $\{1,5,12\}$ is not bimonotone.
Therefore, the set of heights $y$ for which $h_{X,y}$ is supermodular is not convex. This in particular implies that the configuration $X$ is not $\mtp$-tidy. Further, $X$ is also an example of a configuration that is LLC-tidy (as
can be checked from Definition \ref{defn:l.extension}) but not $\mtp$-tidy.\hfill\qed
\end{example}

\subsection{Computing  the $\mtp$ MLE for non-tidy configurations}\label{sec:non-tidy}
Since rational $X$ are always LLC-tidy, we focus on configurations that are not $\mtp$-tidy. In this case, one may hope to find a finite set $X' \supseteq X$ such that $X'$ is $\mtp$-tidy. However, Conjecture~\ref{conj_tidy} suggests that this is difficult.  Algorithm~\ref{defn:conjecture.construction} provides a finite set $X' \supseteq X$
together with a particular candidate bimonotone subdivision $\Delta'$ on $X'$. Though $X'$ is not necessarily tidy, we conjecture that the optimization problem \eqref{tidy_opt_problem} solved over tent functions on $X'$ which induce a coarsening of the subdivision $\Delta'$ yields the solution to the $\mtp$ MLE problem. The intuition for Algorithm~\ref{defn:conjecture.construction} is to add a minimal number of points to $\Delta$ so that the new subdivision $\Delta'$ is bimonotone. 

{
		\begin{algorithm}[h]
		\label{defn:conjecture.construction}
		\caption{Computing a candidate MLE bimonotone subdivision}
		\LinesNumbered
		\DontPrintSemicolon
		\SetAlgoLined
		\SetKwInOut{Input}{Input}
		\SetKwInOut{Output}{Output}
		\Input{$n$ samples $X\subset\mathbb R^d$ with $X=\overline{X}$ and weights $w\in\mathbb R^n$.}
		\Output{A finite set $X' \supseteq X$ and a bimonotone subdivision $\Delta'$ on $X'$.}
		\BlankLine
		Compute the logarithm of the log-concave MLE (a tent function). Let $\Delta$ be the subdivision it induces;
		\;
		If $\Delta$ is a bimonotone subdivision, return $X' = X$, $\Delta' = \Delta$;
		\;
		If $\Delta$ is not bimonotone, compute the hyperplanes spanned by each of the bimonotone facets of $\Delta$, and intersect conv$(X)$ with those hyperplanes; call this new subdivision $\Delta'$ and its vertices $X'$;\; 
		        Output $(X',\Delta')$. 
		        	\end{algorithm}
}

\begin{conjecture}\label{conj_main} 
The MTP$_2$ log-concave MLE is a piecewise linear function whose subdivision is $\Delta'$ or any subdivision refined by $\Delta'$.
\end{conjecture}

This conjecture can be reformulated as follows. 
\begin{conjecture}\label{conj_simple} Let $X'$ be the output of Algorithm~\ref{defn:conjecture.construction}. Let $\hat y$ be the solution to the optimization problem
	\begin{equation}
	\label{tidy_opt_problem.3}
	\text{minimize}\quad - w\cdot y + \int_{\mathbb R^d} exp(h_{X',y}(z))dz \quad \,\,
	\text {s.t.}\quad\  y\in\mathcal S, 
	\end{equation}
    where  $w\in \mathbb R^{|X'|}$ assigns the original weights to the points in $X$, and weight 0 to all points in
     $X' \backslash X$. Then, the tent function $h_{X', \hat y}$ is supermodular.
\end{conjecture}

We now justify why Conjectures~\ref{conj_main} and \ref{conj_simple} are equivalent. The former clearly implies the latter.
 Now, suppose that Conjecture~\ref{conj_simple} is true. Let $\exp(f)$ be the $\mtp$ log-concave MLE, and let $\hat y$ be the optimal heights in \eqref{tidy_opt_problem.3}. Let $y = f(X')$ be the heights induced by $f$. Let $c \geq 1$ be the constant such that $\int c \cdot h_{X', y} = 1$. Then, $\ell(\exp(f))\leq \ell(\exp(h_{X', y + \log(c)})) \leq \ell(\exp(h_{X', \hat y}))$, where equality holds if and only if $f = h_{X', \hat y}$. Since $\exp(f)$ is the MLE, then equality should hold, and therefore, $f$ is a tent function.

We now give some intuition for Conjecture~\ref{conj_simple}.
By Theorem \ref{thm:BimonotoneSubdivision}, if $h_{X', \hat y}$ is not supermodular, then it induces a non-bimonotone subdivision. Consider a non-bimonotone wall in that subdivision. If it is spanned only by points in $X$, then it should have been taken care of by the inequalities for $\hat y\in\mathcal S$. If it involves some points 
in $X' \backslash X$, then, we should be able to lower the heights at those vertices 
(without changing the likelihood), and get a subdivision with that wall removed. 
But now we would have to shift up the whole tent function so that the integral still equals $1$, 
thereby getting a higher likelihood.

 In the following, we verify Conjecture \ref{conj_main} for our running example.

\begin{example}\label{ex:3d_main}
As in Example~\ref{ex_chain},  fix $d=3, n=5$ and let $X$ consist of
$$x_1 = (0,0,0),\, x_2 = (6,0,0), \,x_3 = (6,4,0),\, x_4 = (8,4,2),\, x_5 = (6,4,1.5).$$
We apply Algorithm \ref{defn:conjecture.construction} to
$\, w = \frac1{28}(15,1,1,1,10)$. The output is
 $X' = X\cup\{x_6,x_7\}$, where $x_6=(6, 3, 1.5)$, $x_7=(7.5, 4, 1.5)$.
The solution to (\ref{tidy_opt_problem.3})~is
$$y = (2.95,\, -22.05,\, -14.08,\,  -5.16,\,   0.40,\,  -2.52,\,  -6.47)\,\in\,\mathbb R^7.$$
Let $\exp(f)$ be the log-concave MTP$_2$ MLE and let $\exp(\phi)$ be the log-concave MLE. Using the software {\tt LogConcDead}~\cite{CGS}, we can compute $\phi$ and show that it induces the following triangulation of conv$(X)$ with three tetrahedra:
\begin{align}\label{eq:lcd_triang}\{x_1, x_2,x_3,x_5\},\{x_2,x_5,x_3,x_4\},
\{x_2,x_5,x_4, x_1\}
\end{align}
The interior faces of this triangulation are
$$\{x_2, x_5, x_1\},\{x_2, x_5, x_3\}: \text{bimonotone, }\quad\{x_2, x_5, x_4\}: \text{not bimonotone. }$$

\begin{center}
\includegraphics[width=0.9\textwidth]{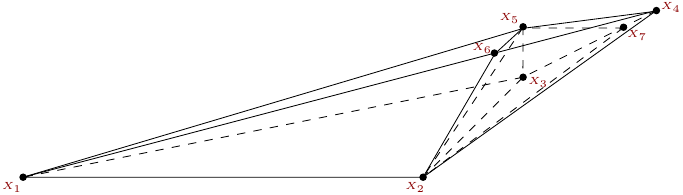}
\end{center}
Consider the hyperplanes spanned by the two interior bimonotone faces:
$$\text{affine span}\{x_2, x_5, x_1\}\,\text{ and }\, \text{affine span}\{x_2, x_5, x_3\}.$$
Intersecting these with the boundary of conv$(X)$ yields $X'=X\cup\{x_6, x_7\}$.
Now, note that the following inequalities hold for $f$:
\small
$$
\frac12f(x_6) + \frac12f(x_7)\geq \frac12 f(x_5) + \frac12f(7.5,3,1.5) \geq \frac12 f(x_5) + \frac38f(x_4) + \frac18f(x_2),
$$
\normalsize
where the first inequality follows by supermodularity, and the second follows by concavity of $f$. Suppose that at least one of these inequalities is strict. Note that for $\phi$ we have the following reverse inequality:
\small
\begin{eqnarray*}
	\frac12\phi(x_6) + \frac12\phi(x_7) &=& \frac38\phi(x_4) + \frac18\phi(x_1) + \frac38\phi(x_4) + \frac18\phi(x_3)\\ &<& \frac12 \phi(x_5) + \frac38\phi(x_4) + \frac18\phi(x_2),
	\end{eqnarray*}
\normalsize
where the inequality follows because $\phi$ induces the triangulation~\eqref{eq:lcd_triang}.
 Hence, there exists $\alpha\in(0,1)$ such that the function
 $\,g_\alpha = \alpha f + (1-\alpha)\phi \,$ satisfies
  \begin{align}\label{eq:5pointsCoplanar}
 \frac12g_\alpha(x_6) + \frac12g_\alpha(x_7) \,\,= \,\,\frac12 g_\alpha(x_5) + \frac38g_\alpha(x_4) + \frac18g_\alpha(x_2).
 \end{align}
 Moreover, $g_\alpha$ gives a higher likelihood than $f$, since $\ell(\exp(g_\alpha)) = \alpha\ell(\exp(f)) + (1-\alpha)\ell(\exp(\phi))$, and the piecewise linear function $h_{X', y'}$ on $X'$, induced by the heights of $g_\alpha$: $y' = g_\alpha(X')$ gives a higher likelihood than $g_\alpha$.
In our experiments, we optimized the likelihood over heights that satisfy the equality~\eqref{eq:5pointsCoplanar}, and the best such tent function induces the bimonotone subdivision $$\{x_1X_2x_3X_5,\, x_1X_2x_5X_6,\, x_2X_3x_5X_7,\, x_2X_5x_4X_6x_7\}.$$ 
Thus, we have found a supermodular tent function with a higher likelihood than $f$, which is a contradiction. As a consequence, equality~\eqref{eq:5pointsCoplanar} has to hold for $f$, and $f$ has to be a tent function,
namely the one computed above. \qed
\end{example}

\section{Discussion}
\label{sec_6}
In this paper, we studied the MLE for nonparametric density estimation under $\mtp$ and LLC, two shape constraints that imply strong forms of positive dependence. These shape constraints are of interest for high-dimensional applications since the MLE exists already for 3 samples (under $\mtp$) and 2 samples (under LLC), irrespective of the number of variables. We proved that the $\mtp$ MLE is a tent function in the two-dimensional or binary setting. We conjectured that this is true in general and we provided an algorithm for computing a candidate MLE. We proved that the MLE under LLC is a tent function when the samples lie in $\mathbb{Q}^d$. It can be computed by solving a finite-dimensional convex optimization problem. Since computations are usually performed in $\mathbb{Q}$ by rounding points in $\mathbb{R}$, in practice, the LLC MLE is always a tent function. Since LLC distributions form a subclass of $\mtp$ distributions, the LLC MLE can always be used as an $\mtp$ estimator, although it might not be the $\mtp$ MLE.

We furnished conditional gradient methods for the finite-dimensional convex optimization problem that computes the MLE under LLC and $\mtp$. Simulations with 55 i.i.d.~samples from a standard Gaussian distribution in $\mathbb{R}^2$ indicate fast convergence of these estimators to the true density. A question for future research is to determine these rates as compared to  the log-concave MLE. In addition, given the recent  work on methods for accelerating the computation of the log-concave MLE from exponential to polynomial time by efficiently computing (sub-)gradients of (an approximation of) the likelihood function~\cite{AxeVal18,DiaSidSte18,RatSch18}, it is natural to ask whether such approaches lead to more efficient algorithms for the MLE under log-concavity and $\mtp$. Developing efficient algorithms for computing the log-concave $\mtp$ MLE is an an interesting avenue for future research.

Since $\mtp$ is implied by many models, including latent tree models 
 in phylogenetics or single factor analysis models in psychology~\cite{MTP2Markov2015}, 
 our work suggests $\mtp$ as a strong enough shape constraint to obtain accurate density
  estimates with relatively few samples, but a large enough class to be of interest for 
  applications. While log-concavity is a natural condition to pair 
  MTP$_2$ with, in future work it would also be interesting to study other constraints.
   Note that obvious constraints such as boundedness from below on a compact set are not sufficient to ensure a bounded likelihood function.

\vspace{0.5cm}

\bibliographystyle{alpha}
\bibliography{non_parametric}

\smallskip
\noindent \small {\bf Authors' addresses:}

\smallskip

\noindent \textit{Corresponding author:} Elina Robeva,
University of British Columbia,
Department of Mathematics, {\tt erobeva@math.ubc.ca}

\noindent Bernd Sturmfels,
 \  MPI-MiS Leipzig, {\tt bernd@mis.mpg.de} \ and \
University of California, Berkeley,  {\tt bernd@berkeley.edu}

\noindent Ngoc Tran,
University of Texas, Austin, 
Department of Mathematics, {\tt ntran@math.utexas.edu}

\noindent Caroline Uhler,
Massachusetts Institute of Technology,
            IDSS and EECS Department,
            {\tt cuhler@mit.edu}.

\bigskip

\appendix

\section{Proofs of results in Section 2}
\label{app_support}

\begin{proof}[Proof of Proposition \ref{prop:llc-extension}]
By definition, $X \subseteq L(X)$, so $L(X) \neq \emptyset$. Since the coordinate of each point in $L(X)$ has a lower and upper bound, $L(X)$ is compact, and thus it is a polytope. Define
$$P^\sharp(X) = \{(x + r\mathbf{1},r): x \in X, r \in \R\} \subset \R^{d+1} .$$ 
Note that
$$ P^\sharp(X) =  \{y \in \R^{d+1}: y_i - y_j \leq c_{ij}, i,j \in [n]\}, $$
where $c_{ij} = \max_{x \in X} (x_i - x_j)$. By \cite[\S 5.5]{murota2003discrete}, this implies that $L(X)$ is $L^\natural$-convex. To see that it is the smallest $L^\natural$-convex set containing $X$, let $P$ be another $L^\natural$-convex set that contains $X$. Define 
$$P^\sharp = \{(x + r\mathbf{1},r): x \in P, r \in \R\} \subset \R^{d+1} .$$ 
 For each pair $i,j \in [n], i \neq j$, $\pi_{ij}(P^\sharp)$ must also be $L$-convex, so it has the following form for some constants $a_{ij} \leq b_{ij}$:
$$ \pi_{ij}(P^\sharp) = \{y \in \R^2: a_{ij} \leq y_1 - y_2 \leq b_{ij} \} .$$
 Since $X \subseteq P$, it follows that $b_{ij} \geq \max_{x \in X} (x_i - x_j)$, and $a_{ij} \leq \min_{x \in X} (x_i - x_j)$. So $\pi_{ij}(P^\sharp) \supseteq \pi_{ij}(P^\sharp(X))$. Since $P^\sharp$ is convex, it follows that $P^\sharp \supseteq \bigcap_{i \neq j} \pi_{ij}^{-1}(\conv(\pi_{ij}(P^\sharp(X)))) = P^\sharp(X)$. This concludes the proof of the first statement. 
For the second statement, suppose that $\tilde{X}$ exists. Since $\tilde{X}$ is $L^\natural$-closed, it satisfies (\ref{eqn:llc.closed}) for some constant $r > 0$. 
Since $X$ is finite, this implies $r(X-v) \in \Z^d$ for any $v \in X$. Thus there is a unique minimal constant $r^\ast$ such that $r(X-v)\in \Z^d$. Define $X' = r^\ast (X-v)$. Then $X'$ is a discrete $L^\natural$-convex set in the sense of \cite[\S 5.5]{murota2003discrete}, so $X' = L(X') \cap \Z^d$. Rearranging gives the RHS of (\ref{eqn:tilde.x}).  
Conversely, suppose that there exists some $v \in X$ and $r > 0$ such that $r(X-v) \subset \Z^d$. Since $X$ is finite, there is a smallest $r$ with this property, denote it $r^\ast$. If $v' \in X$, then $r(v'-v) \in \Z^d$, so $r(X-v') \subset \Z^d$, which implies that $r^\ast$ is independent of the choice of $v$. Finally, define $X'$ as the RHS of (\ref{eqn:tilde.x}). By \cite[5.5]{murota2003discrete}, $X'$ is $L^\natural$-closed. By construction, $X' \supseteq X$. To see that it is the smallest set, suppose that $Y \subset \R^d$ is another $L^\natural$-closed set containing $X$. Since $Y$ is finite, there exists a minimal constant $r(Y) > 0$ such that $r(Y)(X-v) \subset \Z^d$. By minimality of $r^\ast$, we have $r(Y) \geq r^\ast$. So $Y \supseteq X'$, with $Y = X'$ if and only if $r(Y) = r^\ast$. This concludes the proof.
\end{proof}

\begin{proof}[Proof of Proposition \ref{prop_support}]
Let $\hat{f}$ be the log-concave $\mtp$ MLE with support $S \subset \R^d$. Let $\psi_n$ be the objective function
\begin{equation}\label{eqn:psi.n}
\psi_n(f) = \sum_{j=1}^n\log(f(x_j)) - \int_{\mathbb R^d} f(x)dx.
\end{equation}
Since $\hat{f}$ is log-concave and $\mtp$, $S$ is a min-max closed convex set.
Since $\psi_n(\hat{f}) > -\infty$, $X \subseteq S$. As ${\rm MMconv}(X)$ is the smallest min-max closed convex set containing $X$, ${\rm MMconv}(X) \subseteq S$. Now, let $S' = S \backslash {\rm MMconv}(X)$, and $\phi = \hat{f}\mathbf{1}_{\MMconv(X)}$. Since $\hat{f}$ is $\mtp$, $\phi$ is also log-concave $\mtp$. Furthermore, $\psi_n(\phi) \geq \psi_n(\hat{f})$, and equality occurs if and only if $S'$ has measure $0$. As $\hat{f}$ is the maximizer, we must have equality, so ${\rm MMconv}(X) = S$ a.s, as required. For the LLC case, $L(\tilde{X})$ is the smallest $L^\natural$-convex set that contains $X$, so the same proof applies.
\end{proof}

\section{Proof of Theorem 1.1}
\label{app_1}

We here prove  our main theorem. First we derive the $\mtp$ case, then we outline the necessary steps to adapt the proof to the LLC case. Our proof builds upon \cite[Theorem 3.1, Corollary 3.5]{royset2017constrained}, which states that if the set of functions that one optimizes over is equi upper-semicontinuous (equi-usc), then the MLE exists and is consistent. In general, the class of all log-concave $\mtp$ densities on a compact support set is not equi-usc, as seen in the examples in \cite{royset2017constrained}. Our theorem strengthens \cite[Proposition 4.6]{royset2017constrained} by removing the equi-usc assumption and showing that three sample points suffice for existence and uniqueness of the log-concave $\mtp$ MLE. The proof strategy has two parts. First, Corollary \ref{cor:dimension.mtp2} and Lemma \ref{lem:dimension.llc} show that few samples suffice for the support of the MLE to be full-dimensional. Second, Lemma \ref{lem:assumption} employs the argument of \cite{CSS} to show that  one can restrict attention to a smaller subset of log-concave $\mtp$ densities and that this class is totally bounded a.s. Since total boundedness on a compact set is stronger than equi-usc, the conclusion then follows from \cite{royset2017constrained}.}

\begin{lemma}\label{lem:x.not.bim}
Suppose $n \geq 3$ and $d \geq 2$. Almost surely, $X$ is not contained in a bimonotone hyperplane.
\end{lemma}
\begin{proof}
Suppose for contradiction that it is contained in $H = \{x \in \R^d: ax_i - bx_j = c_{ij}\}$,
with $a b \geq 0$. Let $x_1, x_2, x_3$ be points in $X$ that belong to $H$.
Let $\pi_{ij}:\R^d \to \R^2$ be the map $x \mapsto (x_i,x_j)$.
Then $\pi_{ij}(x_1), \pi_{ij}(x_2), \pi_{ij}(x_3) \in \pi_{ij}(H)$. However, $\pi_{ij}(x_1), \pi_{ij}(x_2), \pi_{ij}(x_3)$ are distributed as three i.i.d points from a distribution whose density is absolutely continuous w.r.t. the Lebesgue measure on $\R^2$, and hence they a.s.~do not lie on a line.
\end{proof}

\begin{lemma} \label{lem:non_bim_hyperplane} Let $L\subset \mathbb R^d$ be a linear subspace such that $2\leq \dim L < d$ and which is not contained in any bimonotone hyperplane. Then, there exist points $u,v\in L$ such that $u\wedge v, u\vee v\in \mathbb R^d\setminus L$. 
\end{lemma}

\begin{proof}
Write $L = \{x: Ax = 0\}$, where $A\in\mathbb R^{(n-\dim L)\times d}$. Let $B$ be the reduced-row-Echelon form of $A$, and let $b \in \RR^d$ be the last row of $B$. Since $L$ is not in any bimonotone hyperplane, 
either $b$ has three nonzero entries, or it has two nonzero entries which have the same sign. Assume first that $$b = (\underbrace{0,\ldots,0}_{i-1}, b_i, \ldots, b_j, \ldots, b_k)$$ has three nonzero entries $b_i, b_j, b_k$, where $1\leq i<j<k\leq d$. Let $u\in\mathbb R^d$ be such that $u_i = \text{sign}(b_j)b_j$, $u_j = -\text{sign}(b_j)b_i$, and $u_s = 0$ for all $s >i$, $s\neq j$. For the first $i-1$ entries of $u$, use the reduced-row-Echelon form $B$ so that $Bu = 0$. Similarly, let $v\in\mathbb R^d$ be such that $v_j = -\text{sign}(b_j)b_k$, $v_k = \text{sign}(b_j)b_j$, and $v_s = 0$ for all $s \geq i, s\neq j,k$. For the first $i-1$ entries of $v$, use the reduced-row-echelon form $B$ so that $Bv = 0$. Now, consider the vector $u\wedge v$. We have that $(u\wedge v)_i = 0, (u\wedge v)_j = (-\text{sign}(b_j)b_i)\wedge (-\text{sign}(b_j)b_k)\neq 0, $ and $(u\wedge v)_k = 0$. Therefore, $b\cdot (u\wedge v) = b_j ((-\text{sign}(b_j)b_i)\wedge (-\text{sign}(b_j)b_k))\neq 0$, and thus $B(u\wedge v)\neq 0$. Thus, $(u\wedge v) \not\in L$, and therefore also $(u\vee v)\not\in L$. 

The second case where $b$ has two nonzero entries which have the same sign is similar. This completes the proof.
\end{proof}

Starting from $C^{(0)} = \conv(X)$, we iteratively construct the convex sets 
$$ C^{(i+1)} \,=\, \conv\left(C^{(i)} \cup (C^{(i)}\wedge C^{(i)}) \cup (C^{(i)} \vee C^{(i)})\right)
\,\quad \text{for} \,\,\, i=0,1,2,\ldots. $$
\begin{lemma}\label{lem:c.dimension}
If $\dim C^{(i)} < d$, then, almost surely, $\dim C^{(i+1)}> \dim C^{(i)}$. In particular, almost surely, $\dim C^{(d-2)} = d$.
\end{lemma}
\begin{proof}
 Assume that $\dim C^{(i)} < d$. Consider the affine span of $C^{(i)}$, namely $a + L$, where $L$ is a $\dim(C^{(i)})$-dimensional subspace, and $a$ is a point in the relative interior of $C^{(i)}$ (with respect to its affine span). Then, there exists a ball $B(a, \epsilon)$ of radius $\epsilon$ and center $a$ such that $B(a, \epsilon)\cap (a+L) \subseteq C^{(i)}$. By Lemma \ref{lem:x.not.bim}, $C^{(0)}$ and thus $C^{(i)}$ almost surely are not contained in a bimonotone hyperplane. By Lemma~\ref{lem:non_bim_hyperplane}, there are vectors $u, v\in L$ such that $u\wedge v, u\vee v\not\in L$.   Now, let $c_\epsilon > 0$ be a large enough constant so that the vectors $u_\epsilon = u/c_\epsilon, v_\epsilon=v/c_\epsilon$ have norm less than $\epsilon$. Thus, $a + u_\epsilon, a+v_\epsilon \in B(a, \epsilon) \cap (a+L)$. But note that
\begin{align*}
&(a+u_\epsilon)\wedge(a+v_{\epsilon}) = a + (u\wedge v)/c_\epsilon\not\in (a+L)\quad \text{and} \\ 
&(a+u_\epsilon)\vee(a+v_{\epsilon}) = a + (u\vee v)/c_\epsilon\not\in (a+L).
\end{align*}
Thus, affine span$(C^{(i+1)}) \supsetneq $ affine span$(C^{(i)})$, and so $\dim C^{(i+1)} > \dim C^{(i)}$. Finally, as $n \geq 3$, almost surely, $\dim C^{(0)} \geq 2$,
and hence $\dim C^{(d-2)} = d$.
\end{proof}

\begin{corollary}\label{cor:dimension.mtp2}
Suppose $d \geq 2$. If $n \geq 3$, then $\MMconv(X)$ is full-dimensional a.s..
\end{corollary}

\begin{lemma}\label{lem:assumption} 
For $n \geq 3$ and $T\in\mathbb R$ there exists a constant $M_T$ such that for any MTP$_2$ log-concave density $\exp(f)$, if there is an $x\in\text{ MMconv}(X)$ such that $f(x) > M_T$, then the log-likelihood $\ell(\exp(f), X) := \frac1n \sum_{i=1}^nf(x_i)$ satisfies  $\ell(\exp(f), X)\leq T$.
\end{lemma}

\begin{proof}[Proof of Lemma \ref{lem:assumption}]
Let $\exp(f)$ be a log-concave MTP$_2$ density supported on MMconv$(X)$. Let $$m =\min_{i\in\{1,\ldots,n\}} f(x_i),\quad  M = \max_{x\in\mathbb R^d} f(x), \quad\text{ and }\quad Z \in \text{argmax}_{x\in\mathbb R^d} f(x).$$
By concavity of $f$, we have 
$$f(x)\geq m, \text{ for all } x\in C^{(0)},$$
since $C^{(0)} = \conv(X)$. We are going to show by induction that
$$f(x) \geq 2^{i}m - (2^i - 1)M, \text{ for all } x\in C^{(i)}.$$
We already have the base of the induction when $i=0$. Now, assume that it is true for some $i\geq 0$, and let $x\in C^{(i+1)}$ Then, there exist $a, b\in C^{(i)}$ and $x'\in C^{(i+1)}$ such that $x$ and $x'$ are the minimum and maximum of $a$ and $b$. Thus, since $\exp(f)$ is MTP$_2$, we have that
$$f(x) \geq f(a) + f(b) - f(x') \geq 2(2^im - (2^i-1)M) - M = 2^{i+1}m - (2^{i+1}-1)M,$$
which completes the induction.

Let $m' = 2^{d-2}m-(2^{d-2}-1)M$. Our function $f$ is bounded below by $m'$ on $C^{(d-2)}$ . Observe that for $M$ sufficiently large, we must have that $M - (2^{d-2}m-(2^{d-2}-1)M) > 1$ since $\exp(f)$ is a density and its integral over $C^{(d-2)}$ is at most 1. Now, note that for any $x\in C^{(d-2)}$, we have 
\begin{align*}f\left(Z + \frac1{M-m'}(x-Z)\right) &\geq \frac1{M-m'}f(x) + \frac{M - m'-1}{M-m'}f(Z)\\
&\geq \frac{m'}{M-m'} + \frac{(M-m'-1)M}{M-m'} = M-1,
\end{align*}
where we used concavity of $f$ in the first inequality. Hence, denoting Lebesgue 
measure on $\mathbb R^d$ by $\mu$, we have that
$$\mu\left(\{x: f(x)\geq M-1\}\right) \geq \mu\left(\left\{Z + \frac1{M-m'}C^{(d-2)}\right\}\right) = \frac{\mu(C^{(d-2)})}{(M-m')^d}.$$
Thus,
$$1 \geq \int \exp(f)(x)dx \geq e^{M-1}\frac{\mu(C^{(d-2)})}{(M-m')^d},$$
and, therefore, for $\exp(f)$ to be a density, we need $m' \leq \frac12e^{(M-1)/d}\mu(C^{(d-2)})^{1/d}$ for large $M$. Equivalently, since $m' = 2^{d-2}m - (2^{d-2}-1)M$, we need that
$$m \leq \frac{2^{d-2}-1}{2^{d-2}}M - \frac1{2^{d-1}}e^{(M-1)/d}\mu(C^{(d-2)})^{1/d}.$$
But then the log-likelihood function $\ell(\exp(f), X)$ satisfies
\begin{align*}\ell(\exp(f), X) &\leq \frac{n-1}nM + \frac1n\left(\frac{2^{d-2}-1}{2^{d-2}}M - \frac1{2^{d-1}}e^{(M-1)/d}\mu(C^{(d-2)})^{1/d}\right)\\
& = \frac1n\left(n-1 + \frac{2^{d-2}-1}{2^{d-2}}\right)M - \frac1{n2^{d-1}}e^{(M-1)/d}\mu(C^{(d-2)})^{1/d}.
\end{align*}
Note that as $M\to\infty$, this expression converges to $-\infty$. Therefore, for every $T\in\mathbb R$, there exists $M_T$ such that whenever $\max f(x) \geq M_T$, then, $\ell(\exp(f),X) \leq T$, which completes the proof of Lemma~\ref{lem:assumption}.
\end{proof}

\begin{proof}[Proof of Theorem~\ref{thm_ex_unique} for $\mtp$]
Let $T$ be the likelihood of the uniform density on MMconv$(X)$, and $M_T$ be the corresponding constant in Lemma \ref{lem:assumption}.  We write
$\mathcal{F}$ for the set of all log-concave $\mtp$
probability densities on $\RR^d$ with
$\text{support}(f) = \text{MMconv}(X)$ and $ \sup f \leq M_T $.

By Lemma \ref{lem:assumption} and Proposition \ref{prop_support}, any log-concave $\mtp$ density not in $\widetilde{\mathcal{F}}$ has likelihood strictly less than that of the uniform density on MMconv$(X)$. Thus, it is sufficient to solve \eqref{eq:optproblem} over the set of densities in $\widetilde{\mathcal{F}}$. 
Fix $n \geq 3$. By Lemma \ref{lem:c.dimension}, $\MMconv(X)$ is a.s. full-dimensional. Since $\widetilde{\mathcal{F}}$ is a set of uniformly bounded densities on a compact set, $\widetilde{\mathcal{F}}$ is equi-usc. Therefore, existence and consistency of the MLE follows from \cite[Theorem 3.1, Corollary 3.5]{royset2017constrained}
with $\pi \equiv 0$, $\epsilon = 0$, and $F = \widetilde{\mathcal{F}}$.

To see that the optimizer is a.e. unique, assume that both $f_1,f_2\in\tilde{\mathcal{F}}$ maximize the objective function $\psi_n$ in \eqref{eqn:psi.n}. The normalized geometric~mean
$$g(x) = \frac{\{f_1(x)f_2(x)\}^{1/2}}{\int_{\MMconv(X)} \{f_1(x)f_2(x)\}^{1/2} dx}$$
is in $\tilde{\mathcal{F}}$, with
$$ \begin{matrix}
\psi_n(g) &=& \frac1{2n} \sum_{j=1}^n\log f_1(x_j) + \frac1{2n}\sum_{j=1}^n\log f_2(x_j)
\qquad \qquad \\ & &\qquad -\log\left(\int_{\MMconv(X)}\{f_1(x)f_2(x)\}^{1/2}dx\right)-1  \\
&=& \frac{1}{2}\left(\frac{1}{n}\sum_{j=1}^n\log f_1(x_j) - 1\right) + 
\frac{1}{2}\left(\frac{1}{n}\sum_{j=1}^n\log f_2(x_j) - 1\right) \\
& & \qquad\qquad - \log\left(\int_{\MMconv(X)}\{f_1(x)f_2(x)\}^{1/2}dx\right) \\
&=& \psi_n(f_1) - \log\left(\int_{\MMconv(X)}\{f_1(x)f_2(x)\}^{1/2}dx\right).
\end{matrix} $$
By the Cauchy-Schwarz inequality,  we have
$$\int_{\MMconv(X)} \{f_1(x)f_2(x)\}^{1/2}dx\leq 1,$$ 
so $\psi_n(g)\geq \psi_n(f_1)$, with equality if and only if $f_1=f_2$ almost everywhere.
\end{proof}

\begin{lemma}\label{lem:dimension.llc}
If $\,n,d \geq 2$ then, almost surely, $L(X)$ is full-dimensional.
\end{lemma}

\begin{proof}
Suppose it is not. Since $L(X)$ is $L^\natural$-convex, it must satisfy either $x_i - x_j = c_{ij}$ or $x_i = c_i$. In the first case, define $\pi_{ij}: \R^d \to \R, x \mapsto x_i-x_j \in \R$. Then $\pi_{ij}(L(X))$ is a point, and hence
$\pi_{ij}(X)$ is also a point. But $\pi_{ij}(X)$ are $n$ i.i.d points from a distribution whose density s absolutely continuous w.r.t. the Lebesgue measure on $\R$, since ${\rm supp}(f_0)$ has full dimension. Therefore, a.s. $\pi_{ij}(X)$ is not a singleton for $n \geq 2$, a contradiction. In the second case, define $\pi_i: \R^d \to \R, x \mapsto x_i \in \R$, and the applies a similar argument.
\end{proof}
\begin{proof}[Proof of Theorem~\ref{thm_ex_unique} for LLC]
For sets $C,C' \subset \R^d$, we set $C \oplus C' := \{(x - \alpha \mathbf{1}) \wedge y, (x + \alpha \mathbf{1}) \wedge y: \alpha \in \R\, x \in C, y \in C'\}$. 
We also replace the sets $C^{(i)}$ by $D^{(0)} = \text{conv}(X)$, and
$$ D^{(i+1)} = \conv(D^{(i)} \cup (D^{(i)} \oplus D^{(i)})). $$
The analogue of Lemma \ref{lem:assumption} for LLC has an identical proof, with the $D^{(i)}$ replacing the sets $C^{(i)}$. Similarly, the proof of Theorem 1.1 for the LLC case proceeds in
the same way as the proof for the $\mtp$ case, with $\tilde{\mathcal{F}}$ replaced~by
\begin{align*}
\mathcal{F} = \{f: \R^d \to [0,\infty):\; & f \mbox{ is log-concave LLC},  \\
&\int f = 1, \text{support}(f) = L(X), \sup f \leq M_T \}.
\end{align*}
This completes the proof of Theorem~\ref{thm_ex_unique}.
\end{proof}

\section{Proof of Theorem 3.4}
\label{app_3}

To prove Theorem~\ref{thm:tidy} for $d=2$, we use 

\begin{lemma}\label{lem:diagonal} 
	Let $f:\mathbb R^2\to \mathbb R$ be concave and piecewise linear on 
a polyhedral complex	$\Delta$, and let $u,v \in {\rm supp}(f)$.
	If $u \vee v$ and $u \wedge v$ lie in the same cell of $\Delta$ then $f$
	is supermodular on the four points,~i.e.
	\begin{equation}
	\label{eq:supon4}
	f(u)+f(v)\,\leq \,f(u\wedge v)+f(u \vee v). 
	\end{equation}
\end{lemma}

\begin{proof}
	Let $m$ be the common midpoint of the segments between $u$ and $v$
	and between $u \vee v$ and $u \wedge v$. The``same face'' hypothesis implies that
	the right hand side of (\ref{eq:supon4}) is equal to $2f(m)$.
	Since $f$ is concave, the left hand side of
	(\ref{eq:supon4}) is bounded above by $2f(m)$.
\end{proof}

\begin{proof}[Proof of Theorem~\ref{thm:tidy}]
When $X = \{0,1\}^d$, if $y$ is a vector of tight $\mtp$ heights, then  $h_{X,y}$ equals the Lov\'{a}sz extension \cite[\S 7]{murota2003discrete}, which is $L^\natural$-concave, and thus it is also $\mtp$. The case $X \subseteq \prod_{i=1}^d\{a_i,b_i\}$ reduces to the discrete cube by an invertible map in each coordinate that preserves the bimonotone property.  Now we will prove the case when $X \subset \R^2$ and min-max closed. Note that a polygon in $\R^2$ is not bimonotone if and only if it has an edge with a negative slope. Let $y$ be a set of tight $\mtp$ heights on $X$ and let $h_{X,y}$ be the corresponding tent function. Suppose for contradiction that $h_{X,y}$ is not supermodular. By Theorem~\ref{thm:BimonotoneSubdivision}, there exists an edge of $h_{X,y}$ with negative slope. Let $x,x' \in X$ be the two vertices of this edge. Since $x-x'$ has negative slope, $x$ and $x'$ are not comparable in the lexicographic ordering of $\R^d$. Since $X$ is min-max closed, $x\wedge x',x\vee x' \in X$. As $y$ is $\mtp$,
\begin{equation}\label{eqn:y.leq}
y(x') + y(x) \leq y(x \wedge x') + y(x \vee x'). 
\end{equation}
Since $[x,x']$ is an edge of the regular subdivision, and since $h_{X,y}$ is concave,
\begin{equation}\label{eqn:h.geq}
h_{X,y}(x') + h_{X,y}(x) > h_{X,y}(x \wedge x') + h_{X,y}(x \vee x'). 
\end{equation}
Since $y$ is concave, $h_{X,y}$ equals $y$ at all points of $X$. Therefore, \eqref{eqn:y.leq} and \eqref{eqn:h.geq} cannot simultaneously hold. This is the desired contradiction. 
\end{proof}

\section{Proofs for Section 4}
\label{app_4}

\begin{lemma}\label{lem:mmconv} 
For a finite set $X \subset \R^d$, let 
$$ \begin{matrix} & \min(X) = (\min_{x \in X}x_i: i = 1, \dots, d)\, \in\, \R^d, \\ \mbox{ and } 
& \max(X) = (\max_{x \in X}x_i: i = 1, \dots, d) \,\in\, \R^d. \end{matrix} $$
For any set $X\subset\mathbb R^2$, its min-max convex hull equals
$$\MMconv(X) = \conv(X\cup\{\min(X), \max(X)\}).$$
\end{lemma}

\begin{proof}[Proof of Lemma \ref{lem:mmconv}] First, note that $\conv(X\cup\{\min(X), \max(X)\})\subseteq\text{MMconv}(X)$. Now, we will show that the sides of the polygon $\conv(X\cup\{\min(X), \max(X)\})$ are bimonotone. By Corollary~\ref{cor:bimonotone}, this implies the claim.

First consider the sides of the polygon located above the diagonal between $\min(X)$ and $\max(X)$. Traveling from $\min(X)$ to $\max(X)$, the sides are either vertical or have nonnegative slope. This is because  they have first coordinate bigger than or equal to that of $\min(X)$, and $\conv(X\cup\{\min(X), \max(X)\})$ is convex. Similarly, the sides below the diagonal between $\min(X)$ and $\max(X)$ are either vertical or have nonnegative slope. Therefore, all of the sides of $\conv(X\cup\{\min(X), \max(X)\})$ are bimonotone.
\end{proof}

\begin{lemma}
Algorithm \ref{alg:supermodular_heights.2} outputs the set of inequalities that defines the set of tight $L^\natural$-concave heights for any $L^\natural$-tidy set $X \subset \Z^d$.
\end{lemma}
\begin{proof}
Let $\sharp$ be the operation that identifies $\R^d$ with the subspace of $\R^{d+1}$ defined by $\sum_{i=0}^d x_i = 0$. By definition, $\mathcal{S}$ contains all inequalities of the form $y(x) + y(x') - y(x \vee x') - y(x \wedge x') \geq 0$ for pairs $x,x' \in \tilde{X}^\sharp$. By \cite[Proposition 7.5]{murota2003discrete}, it is sufficient to include these inequalities
 for all pairs $x,x' \in \tilde{X}^\sharp$ such that $\max_i |x_i - x'_i| = 1$. To avoid listing the same 
 inequality multiple times, Algorithm \ref{alg:supermodular_heights.2} goes through 
 each $x \in \tilde{X}^\sharp$ listing all inequalities where $x$ is the minimal point amongst the four points involved.  
\end{proof}

\section{Proof of Theorem 5.1}
\label{app:thmBimonotone}

\begin{proof}[Proof of Theorem~\ref{thm:BimonotoneSubdivision}]
	The only-if direction is simpler, and we prove it first. Let $P$ be a cell of $\Delta$. 
	Suppose that $P$ is not min-max closed. Then we can find $u,v\in P$ such that
	 $u\wedge v \not\in P$ or $u\vee v \not\in P$.  Since $P$ is a convex polyhedron, we have
	$\frac{1}{2}(u+v)\in P$, and it follows that
	$$f(u) + f(v) \,=\, 2 \cdot f(\frac{1}{2} (u+v)) \,>\, f(u\wedge v) + f(u\vee v).$$
	The strict inequality follows because $f$ is concave, but not linear on the segment between $u\wedge v$ and $u\vee v$. 
	This contradicts the assumption that $f$ is supermodular. Hence the polyhedron $P$ is min-max closed and by Corollary~\ref{cor:bimonotone},  $P$ is defined by bimonotone linear inequalities.

	We now prove the if-direction. Suppose that each cell of $\Delta$ is defined by bimonotone linear inequalities.
	Let $u,v \in |\Delta|$. Our goal is to show that $f(u)+f(v)\leq f(u\wedge v) + f(u\vee v)$.
	Let $L$ be the two-dimensional plane in $\RR^d$ that contains $u,v, u\wedge v$ and $u\vee v$. After relabeling
	we can assume that $u_i\geq v_i$ for $i=1,\ldots, k$, and $u_i < v_i$ for $i=k+1,\ldots, d$. The plane equals
	$$ \begin{matrix} L \quad =\quad u \wedge v & + &
	\text{span}\bigl\{(0,\dots, 0, u_{k+1}-v_{k+1}, \dots, u_d-v_d)^T, \\  
	&& \quad (v_1-u_1,\dots, v_k-u_k, 0, \dots, 0)^T \bigr\}. \end{matrix} $$
	From this representation we see that $L$ is defined by $d-2$ linear equations that are bimonotone, namely
	$k-1$ in the first $k$ coordinates, and $d-k-1$ in the last $d-k$ coordinates.
	We identify $L$ with $\RR^2$ via the following order-preserving affine-linear isomorphism  $\RR^2 \rightarrow L$:
	$$ \begin{matrix} (\alpha,\beta) & \mapsto &
	u \wedge v\,\,+\,\, \alpha(0,\dots, 0, u_{k+1}-v_{k+1},\dots, u_d-v_d)^T
	\\ &&  \qquad + \,\,\beta(v_1-u_1,\dots, v_k-u_k,0, \dots, 0)^T. \end{matrix} $$
	This isomorphism preserves the property of a linear function to be bimonotone. Indeed, suppose $\ell$
	is bimonotone on $\RR^d$ and does not vanish on $L$.
	Let $\ell(z) = c + a_iz_i + a_j z_j$, where $a_ia_j \leq 0$. 
	The restriction of $\ell$ to $L$ is the following affine-linear form in the coordinates $\alpha,\beta$:
    \small
	$$ \begin{matrix}
	\ell'(\alpha,\beta)= \\
     = c \,+\, a_i \cdot ((u \wedge v)_i + \alpha(0 \vee (u_i-v_i))) 
	\,+\, a_j \cdot ((u \wedge v )_j + \beta(0\vee (v_j-u_j))) \\
	 = (c+a_i( u \wedge v)_i + a_j ( u \wedge v)_j) + (a_i(0\vee (u_i-v_i))) \cdot \alpha + (a_j(0\vee(v_j-u_j))) \cdot \beta.
	\end{matrix}
	$$
    \normalsize
	This function is bimonotone because $\,(a_i(0\vee (u_i-v_i)))\cdot (a_j(0\vee (v_j-u_j)))\leq 0$.
	
	The restriction  of  $f$ to $L \simeq \RR^2$ is a piecewise-linear concave function 
	$f_L$ in $(\alpha, \beta)$,
	whose support is the restriction $\Delta|_L$ of $\Delta$ to $L$. We apply the reasoning above to each linear function 
	$\ell$ that vanishes on a codimension one cell of $\Delta$. This shows that $f|_L$
	induces a bimonotone subdivision on its support $|\Delta| \cap L$. Our claim 
	follows if $f|_L$ is supermodular. This reduces the proof of
	Theorem~\ref{thm:BimonotoneSubdivision} to the special case $d=2$, where we know it already.
\end{proof}

\bigskip

\end{document}